\documentclass[11pt, oneside]{article}   	
\usepackage[margin=1in]{geometry}              		
\usepackage{graphicx}				
								
\usepackage{amssymb}

\usepackage{amsthm}
\usepackage{amssymb}
\usepackage{amsfonts}
\usepackage{hyperref}
\usepackage{xcolor}
\usepackage{afterpage}
\usepackage{extarrows}
\usepackage{amsmath}
\usepackage{enumitem}
\usepackage{verbatim}
\usepackage{csquotes}
\usepackage{placeins}
\usepackage[capitalize]{cleveref}

\usepackage{graphicx}

\newtheorem{thm}{Theorem}
\numberwithin{thm}{section}
\newtheorem{conj}[thm]{Conjecture}
\newtheorem{prop}[thm]{Proposition}
\newtheorem{cor}[thm]{Corollary}
\newtheorem{claim}[thm]{Claim}
\newtheorem{lem}[thm]{Lemma}

\numberwithin{sch}{subsection}

\theoremstyle{definition}
\newtheorem{qn}{Question}

\newtheorem{defin}{Definition}
\numberwithin{defin}{section}

\numberwithin{prot}{subsection}

\newtheorem{rem}[thm]{Remark}
\newtheorem{fact}[thm]{Fact}

\numberwithin{exs}{subsection}
\newcommand{\Z}{\mathbb{Z}}

\newcommand{\R}{\mathbb{R}}
\newcommand{\C}{\mathbb{C}}
\newcommand{\F}{\mathbb{F}}
\newcommand{\prob}{\mathbb{P}}
\newcommand{\E}{\mathbb{E}}

\newcommand{\cL}{\mathcal{L}}

\newcommand{\cS}{\mathcal{S}}

\newcommand{\lamdba}{\lambda}

\newcommand{\Mat}{\operatorname{Mat}}

\newcommand{\spn}{\operatorname{span}}

\newcommand{\vol}{\operatorname{vol}}

\newcommand{\supp}{\operatorname{supp}}

\newcommand{\disc}{\operatorname{disc}}
\newcommand{\herdisc}{\operatorname{herdisc}}

  \newcommand{\poly}{\operatorname{poly}}

  \newcommand{\eps}{\varepsilon}

\newcounter{const}
\makeatletter
\let\ref\@refstar
\makeatother
\title{On the Discrepancy of Random Matrices with Many Columns}
\author{Cole Franks\thanks{Department of Mathematics, Rutgers University, Piscataway, NJ 08854.
Supported in part by Simons Foundation award 332622.}, Michael Saks\footnotemark[1]}

\begin{document}
\maketitle
\abstract{ Motivated by the Koml\'os conjecture in combinatorial discrepancy, we study the discrepancy of random matrices with $m$ rows and $n$ independent columns drawn from a bounded lattice random variable. It is known that for $n$ tending to infinity and $m$ fixed, with high probability the $\ell_\infty$-discrepancy is at most twice the $\ell_\infty$-covering radius of the integer span of the support of the random variable. However, the easy argument for the above fact gives no concrete bounds on the failure probability in terms of $n$. We prove that the failure probability is inverse polynomial in $m, n$ and some well-motivated parameters of the random variable. We also obtain the analogous bounds for the discrepancy in arbitrary norms.

We apply these results to two random models of interest. For random $t$-sparse matrices, i.e. uniformly random matrices with $t$ ones and $m-t$ zeroes in each column, we show that the $\ell_\infty$-discrepancy is at most $2$ with probability $1 - O(\sqrt{ \log n/n})$ for $n = \Omega(m^3 \log^2 m)$. This improves on  a bound proved by Ezra and Lovett (Ezra and Lovett, \emph{Approx+Random}, 2015) showing that the same is true for $n$ at least $m^t$.  For matrices with random unit vector columns, we show that the $\ell_\infty$-discrepancy is $O(\exp(\sqrt{ n/m^3}))$ with probability $1 - O(\sqrt{ \log n/n})$ for $n = \Omega(m^3 \log^2 m).$ Our approach, in the spirit of Kuperberg, Lovett and Peled (G. Kuperberg, S. Lovett and R. Peled, STOC 2012), uses  Fourier analysis to prove that for $m \times n$  matrices $M$ with i.i.d. columns, and $n$ sufficiently large,  the distribution of $My$ for random $y \in \{-1,1\}^n$ obeys a local limit theorem.
}


\tableofcontents

\section{Introduction}
The topic of this paper is combinatorial discrepancy, a well-studied parameter of a set system or matrix with many applications to combinatorics, computer science and mathematics \cite{C00, M09}. Discrepancy describes the extent to which the sets in a set system can be simultaneously split into two equal parts, or two-colored in a balanced way. Let $\cS$ be a collection (possibly with multiplicity) of subsets of a finite set $\Omega$. The $\ell_\infty$-discrepancy of a two-coloring of the set system $(\Omega, \cS)$ is the maximum imbalance in color over all sets $S$ in $\cS$. The $\ell_\infty$-disrepancy\footnote{often just referred to as the discrepancy} of $(\Omega, \cS)$ is the minimum discrepancy of any two-coloring of $\Omega$. Formally,
$$\disc(\Omega, \cS) := \min_{\chi:\Omega \to \{+1, -1\}} \max_{S \in \cS} |\chi(S)|,$$
where $\chi(S) = \sum_{x \in S} \chi(x)$. More generally, the discrepancy of a matrix $M \in \Mat_{m\times n}(\C)$ or $\Mat_{m\times n}(\R)$ is 
$$\disc(M) = \min_{v \in \{+1, -1\}^n} \|Mv\|_\infty .$$ If $M$ is the incidence matrix of the set system $(\Omega, \cS)$, then the definitions agree. Using a clever linear-algebraic argument, Beck and Fiala showed that the discrepancy of a set system $(\Omega, \cS)$ is bounded above by a function of its \emph{maximum degree} $\Delta(\cS) := \max_{x \in \Omega} |\{S \in \cS: x \in S\}|$. If $\Delta(\cS)$ is at most $t$, we say $(\Omega, \cS)$ is \emph{$t$-sparse}. 
\begin{thm}[Beck-Fiala \cite{BF81}]
If $(\Omega, \cS)$ is $t$-sparse, then $\disc(\Omega,\cS) \leq 2t-1$. 
\end{thm}
Beck and Fiala conjectured that $\disc(\cS)$ is actually $O(\sqrt{t})$ for $t$-sparse set systems $(\Omega, \cS)$. Their conjecture would follow from the following stronger conjecture due to Koml\'os:

\begin{conj}[Koml\'os Conjecture; see \cite{Sp87}]
If every column of $M$ has Euclidean norm at most $1$, then $\disc(M)$ is bounded above by an absolute constant independent of $n$ and $m$.
\end{conj}

This conjecture is still open. The current record is due to Banaszczyk \cite{Ban98}, who showed $\disc(M) = O(\sqrt{ \log n})$ if every column of $M$ has norm at most $1$. This implies $\disc(\Omega, \cS) = O(\sqrt{t \log n})$ if $(\Omega, \cS)$ is $t$-sparse. 

\subsection{Discrepancy of random matrices.}
Motivated by the Beck-Fiala conjecture, Ezra and Lovett initiated the study of the discrepancy of random $t$-sparse matrices \cite{EL16}. Here, motivated by the Koml\'os conjecture, we study the discrepancy of random $m\times n$ matrices with independent, identically distributed columns.

\begin{qn}\label{qn:random_komlos}Suppose $M$ is an $m\times n$ random matrix with independent, identically distributed columns drawn from a vector random variable that is almost surely of Euclidean norm at most one. Is there a constant $C$ independent of
$m$ and $n$ such that for every $\eps > 0$,  $\disc(M) \leq C$ with probability $1 - \eps$ for $n$ and $m$ large enough?
\end{qn}

The Koml\'os conjecture, if true, would imply an affirmative answer to this question. We focus on the regime
where $n \gg m$, i.e., the number of columns is much larger than the number of rows.

A few results are known in the regime $n = O(m)$. The theorems in this direction actually control the possibly larger \emph{hereditary discrepancy}. Define the hereditary discrepancy $\herdisc(M) $ by 
$$ \herdisc(M) = \max_{Y \subset [n]} \disc(M|_Y),$$
where $M|_Y$ denotes the $m \times |Y|$ matrix whose columns are the columns of $M$ indexed by $Y$.

Clearly $\disc(M) \leq \herdisc(M)$. Often the Koml\'os conjecture is stated with $\disc$ replaced by $\herdisc$.

While the Koml\'os conjecture remains open, some progress has been made for \emph{random $t$-sparse matrices}. To sample a random $t$-sparse matrix $M$, choose each column of $M$ uniformly at random from the set of vectors with $t$ ones and $m-t$ zeroes. Ezra and Lovett showed the following:

\begin{thm}[\cite{EL16}] If $M$ is a random $t$-sparse matrix and $n = O(m)$, then 
$\herdisc (M) = O( \sqrt{t \log t})$
  with probability $1 - \exp( - \Omega(t))$.
\end{thm}

The above does not imply a positive answer to \cref{qn:random_komlos} due to the factor of $\sqrt{\log t}$, but is better than the worst-case bound $\sqrt{t \log n}$ due to Banaczszyk.

We now turn to the regime $n \gg m$. It is well-known that if $\disc(M|_Y) \leq C$ holds for all $|Y| \leq m$, then $\disc(M) \leq  2C$ \cite{AS}. However, this observation is not useful for analyzing random matrices in the regime $n \gg m$. Indeed, if $n$ is large enough compared to $m$, the set of submatrices $M|_Y$ for $|Y| \leq m$ is likely to contain a matrix of the largest possible discrepancy among $t$-sparse $m\times m$ matrices, so improving discrepancy bounds via this observation is no easier than improving the Beck-Fiala theorem. The discrepancy of random matrices when $n \gg m$ behaves quite differently than the discrepancy when $n = O(m)$. For example, the discrepancy of a random $t$-sparse matrix with $n = O(m)$ is only known to be $O(\sqrt{t \log t})$, but it becomes $O(1)$ with high probability if $n$ is large enough compared to $m$.
\begin{thm}[\cite{EL16}]\label{el_thm_2}Let $M$ be a random $t$-sparse matrix.
 If $n = \Omega\left(\binom{m}{t} \log \binom{m}{t}\right)$ then 
 $\disc (M) \leq 2$ with probability $1 - \binom{m}{t}^{-\Omega(1)}$.
\end{thm} 

\subsection{Discrepancy versus covering radius}

Before stating our results, we describe a simple relationship between the covering radius of a lattice and a certain variant of discrepancy. We'll need a few definitions.

\begin{itemize}
\item  For $S \subseteq \R^m$, let $\spn_\R S$ denote the
linear span of of $S$, and $\spn_\Z S$ denote the integer span of $S$. 
\item A lattice is a discrete subroup of $\R^m$.  Note that the set $\spn_\Z S$ is a subgroup of $\R^m$, but need not be a lattice.
 If $S$ is linearly
independent or lies inside a lattice, $\spn_\Z S$ {\em is} a lattice. Say a lattice in $\R^m$ is \emph{nondegenerate} if $\spn_\R \cL = \R^m$.

\item
For any norm $\| \cdot \|_*$  on $\R^m$, we write $d_{*}(x,y)$ for the associated distance, and
for $S \subseteq \R^m$, $d_*(x,S)$ is defined to be $\inf_{y \in S} d_*(x,y)$.
\item  The {\em  covering radius} $\rho_*(S)$ of a subset $S$ with respect to the norm $\| \cdot \|_*$ is
$\sup_{x \in \spn_{\R}S} d_*(x,S)$ (which may be infinite.)
\item The discrepancy may be defined in other norms than $\ell_\infty$. If $M$ is an $m \times n$ matrix and $\| \cdot \|_*$ a norm on $\R^m$, define the {\em $*$-discrepancy} $\disc_*(M)$ by 
$$\disc_*(M):= \min_{y \in \{\pm 1\}} \| My\|_*.$$   In particular, $\disc(M)$ is $\disc_{\infty}(M)$.
\end{itemize}

A natural relaxation of $*$-discrepancy is the \emph{odd$_*$ discrepancy}: instead of assigning $\pm1$ to the columns, one could minimize $\|M\mathbf y\|_*$ for $\mathbf y$ with odd entries. By writing each odd integer as $1$ plus an even number, it is easy to see that the odd$_*$ discrepancy of $M$ is equal to 
$$d_* (M \mathbf 1, 2 \cL) \leq 2 \rho_*(\cL).$$
where $\cL$ is the lattice generated by the columns of $M$ and $\mathbf 1$ is the all-ones vector. In fact, by standard argument which can be found in \cite{LSV86}, the maximum odd$_*$ discrepancy of a matrix whose columns generate $\cL$ is sandwiched between $\rho_*(\cL)$ and $2 \rho_*(\cL)$. 

In general, $\disc_*(M)$ can be arbitrarily large compared to the odd$_*$ discrepancy of $M$, even for $m=1, n =2$. If $r \in \Z$ then $M = [2r + 1, r]$ has $\rho_*(\spn_\Z M) = 1/2$ but $\disc_*(M) = r + 1$. However, the discrepancy of a \emph{random} matrix with many columns drawn from $\cL$ behaves more like the odd discrepancy.

\begin{prop}\label{prop:lattice_soft}
Suppose $X$ is a random variable on $\R^m$ whose support generates a lattice $\cL$. 
Then for any $\eps>0$, there is an $n_0(\eps)$ so that for $n > n_0(\eps)$,
 a random $m \times n$ matrix with independent
columns generated from $X$ satisfies 
$$ \disc_*( M) \leq d_* (M \mathbf 1, 2 \cL) \leq 2 \rho_*(\cL)$$
with probability at least $1-\eps$.
\end{prop}
\begin{proof}
Let $S$ be the support of $S$. For every subset $T$ of $S$, let $s_T$ be the sum of the columns of $T$. Let $C$ be large enough that for all $T$, there is an integer combination $v_T$ of elements of $S$ with even coefficients at most $C$ such that $\|v_T - s_T\| \leq d_*(s_T, 2 \cL)$.

Choose $n_0(\eps)$ large enough so that with probability at least $1-\eps$,
if we take $n_0(\eps)$ samples of $X$, every element of $S$ appears at least $C+1$ times. Let $n \geq n(\eps)$ and let $M$ be a random matrix obtained by selecting $n$ columns
according to $X$. With probability at least $1-\eps$ every vector in $S$ appears at least
$C$ times. We claim that if this happens, $\disc_*(M) \leq d_* (M \mathbf 1, 2 \cL).$ This is because if $T$ is the subset of $S$ that appeared an odd number of times in $M$, $d_* (M \mathbf 1, 2 \cL) = d_* (s_T, 2 \cL)$, but because each element of $S$ appears at least $C+1$ times, we may choose $\mathbf y \in \{\pm 1\}^n$ so that $M\mathbf y = s_T - v_T$ for $\|v_T - s_T\| \leq d_*(s_T, 2 \cL)$.
\end{proof}

\subsection{Our results}

The above simple result  says nothing about the number of columns required for $M$ to satisfy
the desired inequality with high probability.  The focus of this paper is on obtaining quantitative upper bounds on the
function $n_0(\varepsilon)$. We will consider the case when $\spn_\Z \supp(X)$ is a lattice $\cL$. The bounds we obtain will be expressed in terms of $m$ and several quantities
associated to the lattice $\cL$, the random variable $X$ and the norm $\|\cdot \|_*$. Without loss of generality, we assume $X$ is symmetric, i.e. $\Pr[X = x] = \Pr[X = -x]$ for all $x$.
For a real number $L > 0$ we write $B(L)$ for the set of points in $\R^m$ of (Euclidean) length at most $L$.

\begin{itemize}
\item The $\|\cdot\|_*$ covering radius $\rho_*(\cL)$.
\item The {\em distortion} $R_*$  of the norm $\|\cdot\|_*$, which is defined to be maximum
Euclidean length of a vector $x$ such that $\|x\|_*=1$. For example, $R_{\infty}=\sqrt{m}$.
\item The determinant $\det \cL$ of the lattice $\cL$, which is the determinant of any matrix whose columns form a basis of $\cL$.
\item The determinant $\det \Sigma$, where $\Sigma=\E[XX^\dagger]$ is the $m \times m$ 
covariance matrix of $X$.
\item The smallest eigenvalue $\sigma$ of $\Sigma.$
\item The maximum Euclidean length $L=L(Z)$ of a vector in the support of $Z=\Sigma^{-1/2}X$.
\item A parameter $s(Z)$ called the \emph{spanningness}.  The definition of this crucial parameter is technical and
is given in \cref{sec:proof_overview}; roughly speaking, it is large if $Z$ is not heavily concentrated near some proper sublattice of $\cL$.
\end{itemize}

We now state our main quantitative theorem about discrepancy of random matrices.  
\begin{thm}[main discrepancy theorem]\label{thm:lattice}
Suppose $X$ is a random variable on a nondegenerate lattice $\cL$. Let $\Sigma:=\E XX^\dagger$ have least eigenvalue $\sigma$. Suppose $\supp X \subset \Sigma^{1/2} B(L)$ and that $\cL = \spn_\Z \supp X$. If $n \geq N$ then
$$ \disc_*( M) \leq d_* (M \mathbf 1, 2 \cL) \leq 2 \rho_*(\cL)$$
with probability at least
$$1 - O\left( L\sqrt{\frac{\log n}{n}}\right).$$ Here $N$, given by \cref{eq:ndisc} in \cref{sec:local}, is a polynomial in the quantities $m$, $s(\Sigma^{-1/2} X)^{-1}$, $L$, $R_*$, $\rho_*(\cL)$, and $\log \left(\det \cL/\det \Sigma\right)$.
\end{thm}

\begin{rem}[degenerate lattices]Our assumption that $\cL$ is nondegenerate is without loss of generality; if $\cL$ is degenerate, we may simply restrict to $\spn_\R \cL$ and apply \cref{thm:lattice}. Further, the assumptions that $\cL = \spn_\Z \supp X$ and $\cL$ is nondegenerate imply $\sigma > 0$. \end{rem}

\begin{rem}[weaker moment assumptions] Our original motivation, the K\'omlos conjecture, led us to study the case when the random variable $X$ is bounded. This assumption is not critical.
We can prove a similiar result under the weaker assumption that $(\E \|X\|_2^\eta)^{1/\eta} = L< \infty$ for some $\eta > 2$. The proofs do not differ significantly, so we give a brief sketch in \cref{sec:moments}.
\end{rem}

Obtaining bounds on the spanningness is the most difficult aspect of applying \cref{thm:lattice}. We'll do this for random $t$-sparse matrices, for which we extend \cref{el_thm_2} to the regime $n = \Omega( m^3 \log^2 m)$. For comparison, \cref{el_thm_2} only applies for $n \gg \binom{m}{t} $, which is superpolynomial in $m$ if $\min(t, m-t) = \omega(1)$.

\begin{thm}[discrepancy for random $t$-sparse matrices]\label{thm:tsparse} Let $M$ be a random $t$-sparse matrix. If $n =\Omega(m^3 \log^2 m)$ then 
$$\disc(M) \leq 2$$
with probability at least $1 - O\left(\sqrt{\frac{m\log n}{n}}\right).$
\end{thm}
\begin{rem}
We refine this theorem later in \cref{thm:tsparse2} of \cref{sec:tsparse} to prove that the discrepancy is, in fact, usually $1$. 
\end{rem}
\refstepcounter{const}\label{const:disc1}
Using analogous techniques to the proof of \cref{thm:lattice}, we also prove a similar result for a non-lattice distribution, namely the matrices with random unit vector columns.
\begin{thm}[random unit vector discrepancy]\label{thm:unit_disc}
Let $M$ be a matrix with i.i.d random unit vector columns. If $n = \Omega(m^3 \log^2m),$
then $$\disc M = O(e^{-\sqrt{\frac{n}{m^3}}})$$
with probability at least $1 - O\left(L\sqrt{\frac{\log n}{n}}\right)$.
\end{thm}
\refstepcounter{const}\label{cubeconst}
\refstepcounter{const}\label{expconst}

One might hope to conclude a positive answer to \cref{qn:random_komlos} in the regime $n \gg m$ from \cref{thm:lattice}.  This seems to require the following weakening of the Koml\'os conjecture: 

\begin{conj}\label{qn:lattice_komlos}
There is an absolute constant $C$ such that for any lattice $\cL$ generated by unit vectors, $\rho_\infty(\cL) \leq C$.
\end{conj}

\subsection{Proof overview}\label{sec:proof_overview}

In what follows we focus on the case when $X$ is isotropic, because we may reduce to this case by applying a linear transformation. The discrepancy result for the isotropic case, \cref{thm:identity_discrepancy}, is stated in \cref{sec:local}, and \cref{thm:lattice} is an easy corollary. We now explain how the parameters in \cref{thm:lattice} arise.

The theorem is proved via \emph{local central limit theorems} for sums of vector random variables.
Suppose $M$ is a fixed $m \times n$ matrix with bounded columns and consider the distribution over $M\mathbf v$ where $\mathbf v$ is chosen uniformly at random from $(\pm 1)^n$.   Multidimensional versions of the central limit theorem imply that this distibution is approximately
normal.  We will be interested in local central limit theorems, which provide precise estimates on the probability that
$M\mathbf v$ falls in a particular region.  By applying an appropriate local limit theorem to a region around the origin, we hope to
show that the probability of being close to the origin is strictly positive, which implies that there is a $\pm 1$ assignment of small discrepancy.

We do not know suitable local limit theorems that work for all matrices $M$.  We will consider random matrices of the form $M=M^X(n)$, where $X$ is a random variable taking values in some lattice $\cL \subset \R^m$,
and $M^X(n)$ has $n$ columns selected independently according to $X$.   We will show that, for suitably large $n$
(depending on the distribution $X$), such a random matrix will, with high probability, satisfy a local limit
theorem. The relative error in the local limit theorem
will decay with $n$, and our bounds will provide quantitative information on this decay rate.  
In order to understand our bounds, it helps to understand what properties of $X$ cause
the error to decay slowly with $n$.

We'll seek local limit theorems that compare $\Pr_{\mathbf y}[M \mathbf y = w]$ to something proportional to $e^{- \frac{1}{2} w^\dagger (M M^\dagger)^{-1} w}$. One cannot expect such precise control if the lattice is very fine. If the spacing tends to zero, we approach the situation in which $X$ is not on a lattice, in which case the probability of expressing any particular element could always be zero! In fact, in the nonlattice situation the covering radius can be zero but the discrepancy can typically be nonzero. For this reason our bounds will depend on $\log (\det \cL)$ and on $L$.

We also need $\rho_*(\cL)$ and the distortion $R_*$ to be small in order to ensure $e^{- \frac{1}{2} w^\dagger (M M^\dagger)^{-1} w}$ is not too small for some vector $w$ that we want to show is hit by $My$ with positive probability over $y$.


Finally, we need that $X$ does not have most of its mass on or near a smaller sublattice $\cL'$. This is the role of spanningness, which is analogous to the spectral gap for Markov chains. Since we assume $X$ is symmetric, choosing the columns $M$ and then choosing $y$ at random is the same as adding $n$ identically distributed copies of $X$. Intuitively, this means that if $M$ is likely to have $My$ distributed according to a lattice Gaussian, then the sum of $n$ copies of $X$ should also tend to the lattice Gaussian on $\cL$. If the support of $X$ is contained in a smaller lattice $\cL'$, then clearly $X$ cannot obey such a local central limit theorem, because sums of copies of $X$ are also contained in $\cL'$. In fact, this is essentially the only obstruction up to translations. We may state the above obstruction in terms of the $\emph{dual lattice}$ and the Fourier transform of $X$.

\begin{defin}[dual lattice]\label{defin:dual} If $\cL$ is a lattice, the \emph{dual lattice} $\cL^*$ of $\cL$ is the set 
$$\cL^*= \{z: \langle z, \lambda \rangle \in \Z \textrm{ for all } \lambda \in \cL\}.$$

\end{defin}
The Fourier transform $\widehat{X}$ of $X$ is the function defined on $\theta \in \R^m$ by $\widehat{X}(\theta) = \E [\exp(2 \pi i \langle X, \theta \rangle )]$. Note that $|\widehat{X}(\theta)|$ is always $1$ for $\theta \in \cL^*$. In fact, if $|\widehat{X}(\theta)| =1$ also \emph{implies} that $\theta \in \cL^*$, then the support of $X$ is contained in no (translation of a) proper sublattice of $\cL$! This suggests that, in order to show that a local central limit theorem holds, it is enough to rule out vectors $\theta$ outside the dual lattice with $|\widehat{X}(\theta)| = 1$. 

In this work, the obstructions are points $\theta$ far from the dual lattice with $\E [|\langle \theta, X \rangle \bmod 1|^2]$ small, where $y \bmod 1$ is taken in $(-1/2, 1/2]$. 
However, we know that for $\theta$ very close to the dual lattice we have $|\langle \theta, x \rangle \bmod 1|^2 = |\langle \theta, x \rangle |^2$ for all $x \in \supp X$, so $\E [|\langle \theta, X \rangle \bmod 1|^2]$ is exactly $d(\theta, \cL^*)^2$. The spanningness measures the value of $\E [|\langle \theta, X \rangle \bmod 1|^2]$ where this relationship breaks down.

\begin{defin}[Spanningness for isotropic random variables]\label{defin:spanningness}

Suppose that $Z$ is an isotropic random variable defined on the lattice $\cL$. Let 
$$\tilde{Z}(\theta):= \sqrt{\E [|\langle \theta, Z \rangle \bmod 1|^2]},$$
where $y \bmod 1$ is taken in $(-1/2, 1/2]$, and say $\theta$ is \emph{pseudodual} if $\tilde{Z}(\theta) \leq d(\theta, \cL^*)/2$. Define the \emph{spanningness} $s(Z)$ of $Z$ by 
$$ s(Z) := \inf_{\cL^* \not\ni \;\theta \textrm{ pseudodual}} \tilde{Z}(\theta).$$ It is a priori possible that $s(Z) = \infty$.

\end{defin}

Spanningness is, intuitively, a measure of how far $Z$ is from being contained in a proper sublattice of $\cL$. Indeed, $s(Z) = 0$ if and only if this the case. Bounding the spanningness is the most difficult part of applying our main theorem. Our spanningness bounds for $t$-sparse random matrices use techniques from the recent work of Kuperberg, Lovett and Peled \cite{LKP12}, in which the authors proved local limit theorems for $My$ for non-random, highly structured $M$. Our discrepancy bounds also apply to the lattice random variables considered in \cite{LKP12} with the spanningness bounds computed in that paper; this will be made precise in \cref{lem:lkp} of \cref{sec:span}.

\subsubsection*{Related work}
We submitted a draft of this work in April 2018, and during our revision process Hoberg and Rothvoss posted a paper on arXiv using very similar techniques on a closely related problem \cite{HR18}. They study random $m\times n$ matrices $M$ with independent entries that are $1$ with probability $p$, and show that for $\disc M = 1$ with high probability in $n$ provided $n = \Omega(m^2 \log m)$.  The results are closely related but incomparable:
our results are more general, but when applied to their setting we obtain a weaker bound of $n \geq \Omega(m^3 \log^2m)$. 




\subsubsection*{Organization of the paper}
In \cref{sec:local} we build the technical machinery to carry about the strategy from the previous section. We state our local limit theorem and show how to use it to bound discrepancy. In \cref{sec:tsparse} we recall some techniques for bounding spanningness, the main parameter that controls our local limit theorem, and use these bounds to prove \cref{thm:tsparse} on the discrepancy of random $t$-sparse matrices. In \cref{sec:unit} we use similar techniques to bound the discrepancy of matrices with random unit columns. \cref{sec:proofs} contains the proofs of our local limit theorems.

\subsubsection*{Notation}
If not otherwise specified, $M$ is a random $m\times n$ matrix with columns drawn independently from a distribution $X$ on a lattice $\cL$ that is supported only in a ball $B(L)$, and the integer span of the support of $X$ (denoted $\supp X$) is $\cL$. $\Sigma$ denotes $\E XX^\dagger$. $D$ will denote the Voronoi cell of the dual lattice $\cL^*$ of $\cL$. $\| \cdot \|_2$ denotes the Euclidean norm for vectors and the spectral norm for matrices, and $\| \cdot \|_*$ denotes an arbitrary norm.

Throughout the paper there are several constants $c_1, c_2, \dots$. These are assumed to be absolute constants, and we will assume they are large enough (or small enough) when needed. 

\section{Likely local limit theorem and discrepancy}\label{sec:local}
Here we show that with high probability over the choice of $M$, the random variable $My$ resembles a Gaussian on the lattice $\cL$. We also show how to use the local limit theorem to bound discrepancy.

For ease of reference, we define the rate of growth $n$ must satisfy in order for our local limit theorems to hold. 
\begin{defin}\label{defin:ndef}
Define $N_0 = N_0(m, s(X), L, \det \cL)$ by 
\begin{align}N_0:=c_{\ref*{nlower}} \max\left\{m^2 L^2( \log m + \log L)^2, s(X)^{-4}L^{-2},L^2 \log^2 \det \cL\right\},\label{eq:nlim}
\end{align}
where $c_{\ref*{nlower}} $ is a suitably large absolute constant. 
\end{defin}
A few definitions will be of use in the next theorem.
\begin{defin}\label{def:gauss_psd}
For a matrix $M$, define the lattice Gaussian with covariance $\frac{1}{2} M M^\dagger$ by
$$G_{M}(\lambda) = \frac{2^{m/2}\det(\cL)}{\pi^{m/2}\sqrt{\det(MM^\dagger)}} e^{-2 \lambda^{\dagger}( MM^\dagger)^{-1}\lambda}.$$
For two Hermitian matrices $A$ and $B$, $A \succeq B$ means $A - B$ is positive-semidefinite.
\end{defin}

\begin{thm}\label{thm:lattice_local_limit} Let $X$ be a random variable on a lattice $\cL$ such that $\E XX^\dagger = I_m$, $\supp X \subset B(L)$, and $\cL = \spn_\Z \supp X$. For $n \geq N_0$, with probability at least $1 - c_{\ref*{failconst}} L\sqrt{\frac{\log n}{n}}$ over the choice of columns of $M$, the following two properties of $M$ hold:
 \begin{enumerate}
 \item $MM^\dagger \succeq \frac{1}{2}n I_m$.
\item For all $\lambda \in \cL -  \frac{1}{2} M \textbf 1$,
\begin{align}\left|\Pr_{y_i \in \{\pm 1/2\}} [M\mathbf y = \mathbf\lambda] - G_{M}(\lambda)\right| = G_{M}(0) \cdot  \frac{ 2m^2 L^2}{n}. \label{eq:thm_bound}
\end{align}
where $G_M$ is as in \cref{def:gauss_psd}.

In particular, for all $\lambda \in \cL -  \frac{1}{2} M \textbf 1$ with $e^{-2 \lambda^{\dagger}( MM^\dagger)^{-1}\lambda} > 2m^2 L^2/n$ we have 
$$\Pr_{y \in \{\pm 1/2\}^n}(My = \lambda) > 0.$$
\end{enumerate}
\end{thm}
Equipped with the local limit theorem, we may now bound the discrepancy. We restate \cref{thm:lattice} using $N_0$.

\begin{thm}[discrepancy for isotropic random variables]\label{thm:identity_discrepancy} Suppose $X$ is an isotropic random variable on a nondegenerate lattice $\cL$ with $\cL=\spn_\Z \supp X$ and $\supp X \subset B(L)$. If $n \geq N$ then
$$ \disc_*( M) \leq d_* (M \mathbf 1, 2 \cL) \leq 2 \rho_*(\cL)$$
with probability $1 - c_{\ref*{failconst}} L\sqrt{\frac{\log n}{n}}$, where \refstepcounter{const}\label{ndisc} 
\begin{align}N_1 =c_{\ref*{ndisc}} \max\left\{R^2_* \rho_*(\cL)^2, N_0\left(m, s(X), L, \det \cL\right)\right\} \label{eq:isodisc}\end{align}
 for $N_0$ as in \cref{eq:nlim}.

\end{thm}

\begin{proof} 

By the definition of the covering radius of a lattice, there is a point $\lambda \in \cL -  \frac{1}{2} M \textbf 1$ with $\|\lambda\|_* \leq d_* (\frac{1}{2} M \mathbf 1,  \cL) \leq \rho_{*}(\cL)$. It is enough to show that, with high probability over the choice of $M$, the point $\lambda$ is hit by $My$ with positive probability over $y \in \{\pm 1/2\}^n$. If so, $2y$ is a coloring of $M$ with discrepancy $2 \rho_* (\cL)$.\\

Because $n$ is at least $N_0(m, s(X), L, \det \cL')$,  the events in \cref{thm:lattice_local_limit} hold with probability at least $1 - c_{\ref*{failconst}} L\sqrt{\frac{\log n}{n}}$. We claim that if the events in \cref{thm:lattice_local_limit} occur, then $\lambda$ is hit by $My$ with positive probability. Indeed, by the final conclusion in \cref{thm:lattice_local_limit}, it is enough to show that
$$e^{-2 \lambda^{\dagger}(MM^\dagger)^{-1}\lambda} > 2 m^2 L^2/n.$$
Because $n \geq N$, $e^{-1} \geq 2 m^2 L^2/n$. Thus, it is enough to show $\lambda^{\dagger}(MM^\dagger)^{-1}\lambda < \frac{1}{2}$. This is true; by $MM^\dagger \geq \frac{1}{2} n I_m$, we have $\lambda^{\dagger}(MM^\dagger)^{-1}\lambda \leq 2\|\lambda\|_*^2 \frac{R^2_*}{n} \leq 2\frac{R^2_*}{n} \rho_*(\cL)^2$.
\end{proof}


Now \cref{thm:lattice} is an immediate corollary of \cref{thm:identity_discrepancy}.

\begin{thm}[Restatement of \cref{thm:lattice}]\label{thm:lattice_disc}Suppose $X$ is a random variable on a nondegenerate lattice $\cL$. Suppose $\Sigma:=\E [XX^\dagger]$ has least eigenvalue $\sigma$, $\supp X \subset \Sigma^{1/2} B(L)$, and that $\cL=\spn_\Z \supp X$. If $n \geq N$ then
$$ \disc_*( M) \leq d_* (M \mathbf 1, 2 \cL) \leq 2 \rho_*(\cL)$$
with probability at least $1 - c_{\ref*{failconst}} L\sqrt{\frac{\log n}{n}}$, where \refstepcounter{const}\label{ndisc} 
\begin{align}N =c_{\ref*{ndisc}} \max\left\{\frac{R^2_* \rho_*(\cL)^2}{\sigma}, N_0\left(m, s(\Sigma^{-1/2}X), L, \frac{\det \cL}{\sqrt{\det \Sigma}}\right)\right\} \label{eq:ndisc}\end{align}
 for $N_0$ as in \cref{eq:nlim}.

\end{thm}
\begin{proof}Note that $\sigma > 0$, because $\cL$ is nondegenerate and $\cL = \spn_Z \supp X \subset \spn_\R\supp X$. Thus, $\sigma = 0$ contradicts $\spn_\R\supp X \subsetneq \R^m$.

Let $Z:= \Sigma^{-1/2} X$ so that $\E [ZZ^\dagger] = I_m$; we'll apply \cref{thm:identity_discrepancy} to the random variable $Z$, the norm $\|\cdot\|_0$ given by $\| v \|_0:=\|\Sigma^{1/2} v \|_*$, and the lattice $\cL' = \Sigma^{-1/2} \cL$. The distortion $R_0$ is at most $R_*/\sigma^{1/2}$, the lattice determinant becomes $\det \cL' = \det \cL/\sqrt{\det \Sigma}$, and $\supp Z \subset B(L)$. The covering radius of $\rho_0(\cL')$ is exactly $\rho_*(\cL)$. Since the choice of $N$ in \cref{eq:ndisc} is $N_1$ of \cref{thm:identity_discrepancy} for $Z, \|\cdot \|_0,$ and $\cL'$, we have from \cref{thm:identity_discrepancy} that
$$\disc_*(M) = \disc_0(\Sigma^{-1/2} M)  \leq 2\rho_0(\cL') = 2 \rho_*(\cL)$$
with probability at least $1 - c_{\ref*{failconst}} L\sqrt{\frac{\log n}{n}}$. \end{proof}


\section{Discrepancy of random $t$-sparse matrices}\label{sec:tsparse}


Here we will state our spanningness bounds for $t$-sparse matrices, and before proving them, compute the bounds guaranteed by \cref{thm:lattice}. For $S \in \binom{[m]}{t}$, let $1_S \in \R^m$ denote the characteristic vector of $S$.
\begin{fact}[random $t$-sparse vector]\label{fact:tsparse} A \emph{random $t$-sparse vector} is $1_S$ for $S$ drawn uniformly at random from $\binom{[m]}{t}$. Let $X$ be a random $t$-sparse vector with $0 < t < m$. The lattice $\cL \subset \Z^m$ is the lattice of integer vectors with coordinate sum divisible by $t$; we have $\rho_\infty(\cL) =1$. Observe that $\cL^* = \Z^m +  \Z \frac{1}{t} \textbf{1}$, where $\textbf{1}$ is the all ones vector. Since $e_1, \dots, e_{m-1}, \frac{1}{t}\textbf 1$ is a basis for $\cL^*$, $\det \cL = 1/\det \cL^*= t$. 
$$\Sigma_{i,j} = \E[X X^\dagger]_{ij} = \left\{ \begin{array}{ccc} 
\frac{t}{m} & i = j \\
 & \\
\frac{t(t-1)}{m(m-1)} & i \neq j \end{array}\right.
$$
The eigenvalues of $\Sigma$ are $\frac{t^2}{m}$ with multiplicity one, and $\frac{t(m-t)}{m(m-1)}$ with multiplicity $m-1$.
\end{fact}
The next lemma is the purpose of the next two subsections. 

\begin{lem}\label{lem:tspanning} There is a constant $\refstepcounter{const} \label{spreadconst} c_{\ref*{spreadconst}}$ such that the spanningness is at least $c_{\ref*{spreadconst}} m^{-1}$; that is,
$$ s(\Sigma^{-1/2}X) \geq c_{\ref*{spreadconst}} m^{-1}.$$

\end{lem}
Before proving this, we plug the spanningness bound into \cref{thm:lattice_disc} to bound the discrepancy of $t$-sparse random matrices.

\begin{proof}[Proof of \cref{thm:tsparse}] If $X$ is a random $t$-sparse matrix, $\|\Sigma^{-1/2} X\|_2$ is $\sqrt{m}$ with probability one. This is because $\E\|\Sigma^{-1/2} X\|_2^2 = m$, but by symmetry $\|\Sigma^{-1/2} x\|_2$ is the same for every $x \in \supp X$. Hence, we may take $L = \sqrt{m}$. By \cref{fact:tsparse}, $\sigma$ is $\frac{t(m-t)}{m(m-1)}$. Now $N$ from \cref{thm:lattice_disc} is at most
\begin{align} c_{\ref*{ndisc}}\max\left\{\underbrace{m \cdot \frac{m(m-1)}{t(m-t)}}_{\frac{R_\infty^2 \rho_\infty(\cL)^2}{\sigma}}, 
\underbrace{m^3\log^2 m}_{m^2 L^2( \log M + \log L)^2}, 
\underbrace{m^3}_{s(X)^{-4}L^{-2}}
\underbrace{m \log^2 t}_{L^2\log^2 \det \cL}\right\},\label{eq:ntsparse}\end{align}
which is $O(m^3 \log^2 m)$. 
\end{proof}

We can refine this theorem to obtain the limiting distribution for the discrepancy.
\begin{thm}[discrepancy of random $t$-sparse matrices]\label{thm:tsparse2} Let $M$ be a random $t$-sparse matrix for $0 < t< m$. Let $Y = \operatorname{B}(m, 1/2)$ be a binomial random variable with $m$ trials and success probability $1/2$. Suppose $n =\Omega(m^3 \log^2 m)$. If $n$ is even, then 
\begin{align*}
\Pr[\disc(M) = 0] &= 2^{-m+1} + O\left(\sqrt{(m/n)\log n}\right)\\
\Pr[\disc(M) = 1] &= (1 - 2^{-m+1}) + O\left(\sqrt{(m/n)\log n}\right)
\end{align*}
and if $n$ is odd then
\begin{align*}
\Pr[\disc(M) =0] &= 0\\
\Pr[\disc(M) = 1] &= \Pr[Y \geq t| Y \equiv t \bmod 2] + O\left(\sqrt{(m/n)\log n}\right)\\
\Pr[\disc(M) = 2] &= \Pr[Y < t| Y \equiv t \bmod 2] + O\left(\sqrt{(m/n)\log n}\right)
\end{align*}
with probability at least $1 - O\left(\sqrt{\frac{m\log n}{n}}\right).$ Note that 
$$\Pr[Y \leq s| Y \equiv t \bmod 2] =2^{-m + 1} \sum_{ k \equiv t \bmod 2}^s \binom{m}{k}.$$
\end{thm}

\begin{proof}[Proof of \cref{thm:tsparse2}] The proof is identical to that of \cref{thm:tsparse} except we do additional work to bound $d_* (M \mathbf 1, 2 \cL)$ instead of just using $2 \rho_*(\cL)$. There are two cases:
\begin{description}
\item[Case 1: $n$ is odd.] The coordinates of $M \mathbf 1$ sum to $nt$. The lattice $2\cL$ is the integer vectors with even coordinates whose sum is divisible by $2t$. Thus, in order to move $M \mathbf 1$ to $2 \cL$, each odd coordinate must be changed to even and the total sum must be changed by an odd number times $t$. The number of odd coordinates has the same parity as $t$, so we may move $M$ to $2\cL$ by changing each coordinate by at most $1$ if and only if the number of odd coordinates is at least $t$.
\item[Case 2: $n$ is even.] In this case, the total sum of the coordinates must be changed by an \emph{even} number times $t$. The parity of the number of odd coordinates is even, so the odd coordinates can all be changed to even preserving the sum of all the coordinates. This shows may move $M$ to $2\cL$ by changing each coordinate by at most $1$, and by at most $0$ if all the coordinates of $M \mathbf 1$ are even. 
 \end{description}
Thus, in the even case the discrepancy is at most $1$ with the same failure probability and $0$ with the probability all the row sums are even, and in the odd case the discrepancy is at most $1$ provided the number of odd coordinates of $M \mathbf 1$ is at least $t$. Observe that the vector of row sums of a $m\times n$ random $t$-sparse matrix taken modulo $2$ is distributed as the sum of $n$ random vectors of Hamming weight $t$ in $\F_2^m$. \cref{lem:markovmixing} below shows that the Hamming weight of this vector is at most $O(e^{-2n/m + 3m})$ in total variation distance from a binomial $\operatorname{B}(m, 1/2)$ conditioned on having the same parity as $nt$. Because this is dominated by $\sqrt{(m/n)\log n}$ for $n \geq m^3 \log^2m$, the theorem is proved. \end{proof}
\begin{lem}[number of odd rows]\label{lem:markovmixing}
Suppose $X_n$ is a sum of $n$ uniformly random vectors of Hamming weight $0 <t<m$ in $\F_2^m$ and $Z_n$ is a uniformly random element of $\F_2^m$ with Hamming weight having the same parity as $nt$. If $d_{TV}$ denotes the total variation distance, then

$$ d_{TV}(X_n,  Z_n) = O(e^{- 2n/m + m}).$$

\end{lem}
\begin{proof} Though we will not use the language of Markov chains, the following calculation consists of showing that the random walk on the group $\F_2^m$ mixes rapidly by showing it has a spectral gap.

Let $X$ be a random element $\F_2^m$ of Hamming weight $t$. Let $f$ be the probability mass function of $X$. Let $h_n$ be the probability mass function of $Z_n$. By the Cauchy-Schwarz inequality, it is enough to show that the probability mass function $f_n$ of the sum of $n$ i.i.d. copies of $X$ satisfies 
\begin{align*}\sum_{x \in \F_2^m} |f_n(x) - h_n(x)|^2 = O(e^{-2n/m})\end{align*}
For $y \in \F_2^m$, let $\chi_y: \F_2^m \to \{\pm 1\}$ be the Walsh function $\chi_y(x) = (-1)^{y \cdot x}$. The Fourier transform of a function $g: \F_2^m \to \R$ is the function $\widehat{g}: \F_2^m \to \R$ given by $ \widehat{g}(y) = \sum_{x \in \F_2^m} g(x)\chi_y(x).$
The function $f_n$ satisfies $\widehat{f}_n = (\widehat{f}\;)^n$. Note that $\widehat{h}_n(\mathbf 0) =\widehat{f}_n(\mathbf 0)= 1$, $\widehat{h}_n(\mathbf 1) = \widehat{f}_n(\mathbf 1) = (-1)^{nt}$, and $\widehat{h} = 0$ elsewhere. By Plancherel's identity,   
\begin{align}\sum_{x \in \F_2^m}|f_n(x) - h_n(x)|^2 &= \E_{y \in \F_2^m} |\widehat{f}(y)^n -  \widehat{h}_n(y)|^2\\
& = \sum_{y \in \F_2^m, y \neq \mathbf0, y \neq \mathbf 1} 2^{-n}|\widehat{f}(y)|^{2n}.\label{walsh}
\end{align}
Now we claim that $|\widehat{f}(y)| \leq 1 - \frac{1}{m}$ for $y \not \in \{\mathbf 0, \mathbf1\}$, which would imply \cref{walsh} is at most 
$ (1 - 1/m)^{2n} \leq e^{-2n/m}.$ Indeed, if the Hamming weight of $y$ is $s$, then $\widehat{f}(y)$ is exactly the expectation of $(-1)^{|S \cap T|}$ where $T$ is a random $t$-set and $S$ a fixed $s$-set. By symmetry we may assume $t \leq s$, and since we are only concerned with the absolute value of this quantity, by taking the complement of $S$ we may assume $s \leq m/2$. We may choose the elements of $T$ in order; it is enough to show that the expectation of $(-1)^{|S \cap T|}$ is at most $1 - 1/m$ in absolute value even after conditioning on the choice of the first $t-1$ elements of $T$. Indeed, the value of $(-1)^{|S \cap T|}$ is not determined by this choice, so the conditional expectation is a rational number in $(-1,1)$ with denominator at most $m$, and hence at most $1 - 1/m$ in absolute value.\end{proof}
We'll now discuss a general method for bounding the spanningness of lattice random variables.


\subsection{Spanningness of lattice random variables}\label{sec:span}
Suppose $X$ is a finitely supported random variable on $\cL$. We wish to bound the spanningness $s(X)$ below. The techniques below nearly identical to those in \cite{LKP12}, in which spanningness is bounded for a very general class of random variables. 

We may extend spanningness for nonisotropic random variables.
\begin{defin}[nonisotropic spanningness]\label{defin:noniso_span}
A distribution $X$ with finite, nonsingular covariance $\E XX^\dagger = \Sigma$ determines a metric $d_X$ on $\R^m$ given by $d_X(\theta_1, \theta_2) = \|\theta_1 - \theta_2\|_X$ where the square norm $\|\theta \|^2_X$ is given by $\theta^\dagger \Sigma \theta = \E[ \langle X , \theta \rangle^2]$. 
Let 
$$\tilde{X}(\theta):= \sqrt{\E [|\langle \theta, X \rangle \bmod 1|^2]},$$
where $y \bmod 1$ is taken in $(-1/2, 1/2]$, and say $\theta$ is \emph{pseudodual} if $\tilde{X}(\theta) \leq d_X(\theta, \cL^*)/2$. Define the \emph{spanningness} $s(X)$ of $X$ by 
$$ s(X) := \inf_{\cL^* \not\ni \;\theta \textrm{ pseudodual}} \tilde{X}(\theta).$$ 
This definition of spanningness is invariant under invertible linear transformations $X \leftarrow AX$ and $\cL \leftarrow A \cL$; in particular, $s(X)$ is the same as $s(\Sigma^{-1/2} X)$ for which we have $\|\theta\|_{\Sigma^{-1/2}X} = \|\theta\|_2$. Hence, this definition extends the spanningness of \cref{defin:spanningness}.
\end{defin}

\subsubsection*{Strategy for bounding spanningness}
Our strategy for bounding spanningness below is as follows: we need to show that if $\mathbf \theta$ is \emph{pseudodual but not dual}, i.e., $0 < \tilde{X}(\theta) \leq d(x, \cL^*)/2$, then $\tilde{X}(\theta)$ is large. We do this in the following two steps.

\begin{enumerate}
\item Find a $\delta$ such that if all $|\langle  x,  \theta \rangle \bmod 1| \leq \frac{1}{\beta}$ for all $x \in \supp X$, then $\tilde{X}(\theta) \geq d_X( \theta, \cL^*)$. Such $\theta$ cannot be pseudodual without being dual.
\item $X$ is $\alpha$-\emph{spreading}: for all $\theta$,
$$\tilde{X}(\theta) \geq \alpha \sup_{x \in \supp X} |\langle  x,  \theta \rangle \bmod 1|$$
\end{enumerate}
Together, if $\theta$ is pseudodual, then $\tilde{X}(\theta) \geq \alpha/\beta$.

To achieve the first item, we use bounded integral spanning sets as in \cite{LKP12}. The following definitions and lemmas are nearly identical to arguments in the proof of Lemma 4.6 in \cite{LKP12}.  
\begin{defin}[bounded integral spanning set]\label{elvee}
Say $B$ is an integral spanning set of a subspace $H$ of $\R^m$ if $B \subset \Z^m$ and $\spn_\R B= H$. Say a subspace $H \subset \R^m$ has a \emph{$\beta$-bounded integral spanning set} if $H$ has an integral spanning set $B$ with 
$\max\{\|b\|_1: b \in B\} \leq \beta.$
\end{defin}
\begin{defin}\label{gammavee}Let $A_X$ denote the matrix whose columns are the support of $X$ (in some fixed order). Say $X$ is \emph{$\beta$-bounded} if $\ker A_X$ has a $\beta$-bounded integral spanning set.
\end{defin}

\begin{lem}\label{dualclose}
Suppose $X$ is $\beta$-bounded. Then either
$$\max_{x \in \supp(X)}|\langle x, \theta \rangle \bmod{1}| \geq \frac{1}{\beta}$$
or 
$$\tilde{X}(\theta) \geq d_X(\theta, \cL^*)$$

\end{lem}
\begin{proof}
To prove Lemma \ref{dualclose} we use a claim from \cite{LKP12} , which allows us to deduce that if $\langle x, \theta \rangle$ is very close to an integer for all $x$ then we can ``round" $\theta$ to an element of the dual lattice to get rid of the fractional parts.
\begin{claim}[Claim 4.12 of \cite{LKP12}]\label{snapto}Suppose $X$ is $\beta$ bounded, and define $r_x := \langle x, \theta \rangle \bmod{1} \in (-1/2, 1/2]$ and $k_x$ to be the unique integer such that $\langle x, \theta \rangle = k_x + r_x$. If 
$$\max_{x \in \supp(X)} |r_x| < 1/\beta$$ then there exists $l \in \cL^*$ with 
$$\langle x, l \rangle = k_x $$
for all $x \in \supp(X)$. 
\end{claim}
Now, suppose $\max_{x \in \supp(X)}|\langle x, \theta \rangle \bmod{1}| = \max_{x \in \supp(X)} |r_x|  < 1/\beta$. By Claim \ref{snapto}, exists $l \in \cL^*$ with $\langle x, l \rangle = k_x$ for all $x \in \supp(X)$.  By assumption,
$$ \tilde{X}(\theta) =  \sqrt{\E(\langle X, \theta\rangle \bmod 1)^2} = \sqrt{\E r_X^2}= \sqrt{\E \langle X,  \theta - l \rangle^2} \geq d_X(\theta, \cL^*),$$
proving Lemma \ref{dualclose}.\end{proof}
In order to apply Lemma \ref{dualclose}, we will need to bound $\tilde{X}(\theta)$ below when there is some $x$ with $|\langle x, \theta \rangle|$ fairly large.
\begin{defin}\label{alphavee}
Say $X$ is $\alpha$-spreading if for all $\theta \in \R^m$,
$$\tilde{X}(\theta) \geq \alpha \cdot \sup_{x \in \supp X}| \langle x, \theta \rangle \bmod{1}|.$$

\end{defin}
Combining \cref{dualclose} with \cref{alphavee} yields the following bound.
\begin{cor}\label{spreadcor} Suppose $X$ is $\beta$-bounded and $\alpha$-spreading. Then $s(X) \geq \frac{\alpha}{\beta}$.
\end{cor}
A lemma of \cite{LKP12} immediately gives a bound on spanningness for random variables that are uniform on their support. 
\begin{lem}[Lemma 4.4 from \cite{LKP12}]\label{lem:lkp} Suppose $X$ is uniform on $\supp X \subset B_\infty( L')$ and for any two elements $x, y \in \supp X$ there is an invertible linear transformation $A$ such that $Ax= y$ and $X = AX$. Then $X$ is $$\Omega \left(\frac{1}{(m \log (L' m))^{3/2}} \right) \textrm{-spreading}.$$
In particular, if $X$ is $\beta$-bounded, then 
$$ s(X) = \Omega \left(\frac{1}{\beta(m \log (L' m))^{3/2}} \right).$$
\end{lem}

\if{false}
\begin{proof}
We must show that if if $\eps^2 < 2 \frac{\alpha}{\beta^2}$ then 
$\inf_{D_X \setminus B_X(\eps)} \E (\langle \theta, X \rangle \bmod 1)^2 \geq \frac{1}{2} \eps^2.$
Suppose $\theta \in D_X \setminus B_X(\eps)$ so that in particular, $d_X(\theta, \cL^*) \geq \eps$. Suppose $\max_{x \in \supp(X)}|\langle x, \theta \rangle \bmod{1}| \geq \frac{1}{\beta}$;  by the assumption that $X$ is $\alpha$-spreading we have
$$\E[(\langle X, \theta \rangle \bmod{1})^2] \geq \frac{\alpha}{\beta^2}.$$
If $\max_{x \in \supp(X)}|\langle x, \theta \rangle \bmod{1}| < \frac{1}{\beta}$, we must be in the second case of Lemma \ref{dualclose} and so
$$\E(\langle X, \theta\rangle \bmod 1)^2 \geq \eps^2.$$
Since $\frac{\alpha}{\beta^2}, \eps^2 \geq \frac{1}{2} \eps^2$, the proof is complete.
\end{proof}
\fi


\subsection{Proof of \cref{lem:tspanning}}

Using the techniques from the previous section, we'll prove \cref{lem:tspanning}, which states that $t$-sparse random vectors have spanningness $\Omega(m^{-1})$. In particular, we'll prove that $t$-sparse random vectors are $4$-bounded and $\Omega(m^{-1})$-spreading and apply \cref{spreadcor}. 


\subsubsection{Random $t$-sparse vectors are $\Omega(m^{-1})$-spreading}
Note that \cref{lem:lkp} gives that $t$-sparse vectors are $\Omega \left(\frac{1}{(m \log (m))^{3/2}} \right) \textrm{-spreading}, $ but we can do slightly better due to the simplicity of the distribution $X$.

In order to show that $t$-sparse vectors are $c$-spreading, recall that we must show that if a \emph{single} vector $1_S$ has $|\langle \theta, 1_S \rangle \bmod 1| > \delta$, then $\E[|\langle \theta, X \rangle \bmod 1|^2] \geq c^2\delta^2$. We cannot hope for $c = o(m^{-1})$, because for small enough $\delta$ the vector $\theta = \delta (\frac{1}{t} 1_{[t]} - \frac{1}{m-t} 1_{m\setminus [t]})$ has $\langle \theta, 1_{[t]} \rangle \bmod 1 = \delta$ but $\E[|\langle \theta, X \rangle \bmod 1|^2] = \frac{1}{m-1} \delta^2$. Our bound is worse than this, but the term in \cref{eq:ntsparse} depending on the spanningness is not the largest anyway, so this does not hurt our bounds. 
 

\begin{lem}\label{sparsespread}
There exists an absolute constant \refstepcounter{const} \label{spreadconst}$c_{{\theconst}}>0$ such that random $t$-sparse vectors are $\frac{c_{\ref*{spreadconst}}}{m}$-spreading.
\end{lem}
\begin{proof}
If $t = 0$ or $t = m$, then $t$-sparse vectors are trivially $1$-spreading. Suppose there is some $t$-subset of $[m]$, say $[t]$ without loss of generality, satisfying $|\langle \theta, 1_{[t]}\rangle \bmod 1| = \delta > 0$. For convenience, for $S \in \binom{[m]}{t}$, define 
$$w(S) := |\langle \theta, 1_{S}\rangle \bmod 1|.$$ We need to show that $w([t]) = \delta$ implies $\E_S w(S)^2 =\Omega(m^{-2}\delta^2)$. To do this, we will define random integer coefficients $\lambda = \left(\lambda_S: S \in \binom{[m]}{t}\right)$ such that
$$1_{[t]} = \sum_{S \in \binom{[m]}{t}} \lambda_S 1_S.$$
Because $|a + b \bmod 1| \leq |a \bmod 1| + |b \bmod 1|$ for our definition of $\bmod 1$, we have the lower bound 
\begin{align}\delta = w([t]) \leq \E_{\lambda} \sum_{S\in \binom{[m]}{t}}  w(S) |\lambda_S|=  \sum_{S \in \binom{[m]}{t}} w(S) \cdot \E_\lambda |\lambda_S|.\label{eq:coeffs}\end{align}
It is enough to show $\E_\lambda |\lambda_S|$ is small for all $S$ in $\binom{m}{t}$, because then $\E [  w(S)]$ is large and 
$$\E[ w(S)^2] \geq \E[ w(S)]^2.$$ We now proceed to define $\lambda$. Let $\sigma$ be a uniformly random permutation of $[n]$ and let $T = \sigma({[t]})$. 
We have 
\begin{align} 1_{[t]} = 1_T + \sum_{i \in {[t]}: i \neq \sigma (i)} e_i - e_{\sigma(i)}, \label{eq:permutation}\end{align}
where $e_i$ is the $i^{th}$ standard basis vector. Now for each $i \in {[t]}: i \neq \sigma (i)$ pick $R_i$ at random conditioned on $i \in R_i$ but $\sigma(i) \not\in R_i$. Then 
\begin{align}e_i - e_{\sigma(i)} = 1_{R_i} - 1_{R_i - i + \sigma(i)}. \label{eq:difference}
\end{align}To construct $\lambda$, first define the indicator $e^U$ by $ e^U_S:= 1_{S =U}$ for $U, S \in \binom{[m]}{t}$, and then define 
$$ \lambda = e^T - \sum_{i \in {[t]}: i \neq \sigma (i)} e^{R_i} - e^{R_i - i + \sigma(i)}.$$
By \cref{eq:permutation} and \cref{eq:difference}, this choice satisfies $\sum \lambda_S 1_S = 1_{[t]}$. 

It remains to bound $\E_\lambda|\lambda_S|$ for each $S$. We have
\begin{align} \E_\lambda |\lambda_S| &\leq \Pr [T = S]\label{eq:tee}\\
&+ \sum_{i = 1}^t \Pr[\sigma(i) \neq i \textrm{ and }R_i = S]\label{eq:rplus}\\
& + \sum_{i = 1}^t \Pr[\sigma(i) \neq i \textrm{ and } R_i - i + \sigma(i) = S].\label{eq:rminus}
\end{align}
since $T$ is a uniformly random $t$-set, $\cref{eq:tee} = \binom{m}{t}^{-1}$. Next we have 
$\Pr[\sigma(i) \neq i \textrm{ and }R_i = S] = \frac{m-1}{m} \Pr[R_i = S]$. However, $R_i$ is chosen uniformly at random among the $t$-sets containing $i$, so 
$$\Pr[R_i = S] =\binom{m-1}{t-1}^{-1} 1_{i \in S} = \frac{m}{t}  \binom{m}{t}^{-1} 1_{i \in S}. $$
Thus $\cref{eq:rplus} \leq  (m-1)  \binom{m}{t}^{-1}.$ Similarly, $R_i - i + \sigma(i)$ is chosen uniformly at random among sets \emph{not} containing $i$, so 
$\Pr[ R_i - i + \sigma(i) = S] =  \binom{m}{t-1}^{-1} 1_{i \not\in S} = \frac{m-t + 1}{t} \binom{m}{t}^{-1} 1_{i \not\in S}$. Thus 
$\cref{eq:rplus} \leq  (m-1)  \binom{m}{t}^{-1}.$ Thus, for every $S$ we have $\E_\lambda|\lambda_S| \leq 2 m \binom{m}{t}^{-1}.$ Combining this with \cref{eq:coeffs} we have 
$$\E[ w(S)^2] \geq \E[ w(S)]^2 \geq (2m)^{-2} \delta^2.$$

We may take $c_{\ref*{spreadconst}} = 1/2$. \end{proof}

\subsubsection{Random $t$-sparse vectors are $4$-bounded}

Recall that $A_X$ is a matrix whose columns consist of the finite set $\supp X = \left\{1_S: S \in \binom{[m]}{t}\right\}.$ We index the columns of $A_X$ by $\binom{[m]}{t}$. 
\begin{lem}\label{4bd}
$X$ is $4$-bounded. That is, $\ker A_X$ has a $4$-bounded integral spanning set.
\end{lem}
Before we prove the lemma we establish some notation. We have $A_X:\R^{\binom{[m]}{t}} \to \R^m$. Let $e_S \in \R^{\binom{[m]}{t}}$ denote the standard basis element with a one in the $S$ position and 0 elsewhere. For $i \in [m]$, $e_i \in \R^m$ denotes the standard basis vector with a one in the $i^{th}$ position and 0 elsewhere.

\begin{defin}[the directed graph $G$]
For  $S, S' \in \binom{[m]}{t}$ we write $S' \to_j S$ if $1 \in S'$, $j \not \in S'$ and
$S$ is obtained by replacing 1 by $j$ in  $S'$.  Let $G$ be the directed graph with $V(G) = \binom{[m]}{t}$ and 
$S'S \in E(G)$ if and only if $S'\to_j S$ for some $j \in S\setminus S'$.  Thus every set containing 1 has out-degree
$m-t$  and in-degree 0 and every set not containing 1 has in-degree $t$ and out-degree 0. 
\end{defin}

The following proposition implies Lemma \ref{4bd}. Note that if $S'\to_j S$, then $1_{S'} - 1_{S} = e_1 - e_j$.
\begin{prop}\label{spanning}
$$\cS = \bigcup_{j = 2}^{m}\{e_{S'} - e_S + e_T - e_{T'}: S' \to_j S \textrm{ and } T' \to_j T\}$$ is a spanning set for $\ker A_X$.
\end{prop}

\begin{proof}[Proof of \cref{spanning}]
Clearly $\cS$ is a subset of $\ker A_X$, because if $S'\to_j S$, then $1_{S'} - 1_{S} = e_1 - e_j$, and so $A_X (e_{S'} - e_S) = 1_{S'} - 1_{S} = e_1 - e_j$. Thus, if $S' \to_j S$ and $T' \to_j T$, $A_X (e_{S'} - e_S + e_T - e_{T'}) = 0$. If $S'\to_j S$, then $A_X (e_{S'} - e_S) = 1_{S'} - 1_{S} = e_1 - e_j$. Thus, if $S' \to_j S$ and $T' \to_j T$, $A_X (e_{S'} - e_S + e_T - e_{T'}) = 0$, so $e_{S'} - e_S + e_T - e_{T'} \in \ker A_X$. \\

Next we try to prove $\cS$ spans $\ker A_X$. Note that $\dim \ker A_X = \binom{m}{t}-m$, because the column space of $A_X$ is of dimension $m$ (as we have seen, $e_1 - e_j$ are in the column space of $A_X$ for all $1 < j \leq m$; together with some $1_S$ for $1 \notin S \in \binom{[m]}{t}$ we have a basis of $\R^m$). Thus, we need to show $\dim \spn_\R \cS$ is at least $\binom{m}{t}-m$.\\

For each $j \in [m]-1$, there is some pair $T_j, T_j' \in \binom{[m]}{t}$ such that $T'_j \to_j T_j$. For $j \in \{2, \dots, m\}$, pick such a pair and let $f_j := e_{T'_j} - e_{T_j}$. As there are only $m-1$ many $f_j$'s, $\dim \spn \{f_j: j \in [m] - 1\} \leq m-1$. By the previous argument, if $S' \to_j S$, then $e_{S'} - e_S  - f_j \in \ker A_X$.  Because $\bigcup_{j = 2}^{m}\{e_{S'} - e_S  - f_j : S'\to_jS\} \subset \cS$, it is enough to show that
$$\dim \spn_\R \bigcup_{j = 2}^{m}\{e_{S'} - e_S  - f_j : S'\to_jS\} \geq \binom{m}{t}-m.$$ 
We can do this using the next claim, the proof of which we delay.
\begin{claim}\label{highdim}
$$\dim\spn_\R\bigcup_{j=2}^m \{e_{S'} - e_S :  S' \to_j S\} = \binom{m}{t}-1.$$
\end{claim}
Let's see how to use \cref{highdim} to finish the proof:
\begin{eqnarray*}
\dim \spn_\R \bigcup_{j = 2}^{m}\{e_{S'} - e_S  - f_j : S'\to_jS\} \geq\\
 \dim \spn_\R \bigcup_{j = 2}^{m}\{e_{S'} - e_S: S'\to_jS\} - \dim\spn_\R\{f_j: 1 \neq j \in [m]\} \geq\\ \binom{m}{t}-1 - (m-1) = \binom{m}{t}-m.
\end{eqnarray*}
The last inequality is by Claim \ref{highdim}.
\end{proof}
Now we finish up by proving \cref{highdim}.
\begin{proof}[Proof of \cref{highdim}]
If a directed graph $H$ on $[l]$ is \emph{weakly connected}, i.e. $H$ is connected when the directed edges are replaced by undirected edges, then $\spn\{e_i -e_j: ij \in E(H)\}$ is of dimension $l-1$. To see this, consider a vector $v\in \spn_\R\{e_i -e_j: ij \in E(H)\}^\perp$. For any $ij \in E$, we must have that $v_i  = v_j$. As $H$ is weakly connected, we must have that $v_i = v_j$ for all $i, j \in [l]$, so $\dim\spn_\R\{e_i -e_j: ij \in E(H)\}^\perp \leq 1$. Clearly $\mathbf 1 \in \spn_\R\{e_i -e_j: ij \in E(H)\}^\perp$, so $\dim\spn_\R\{e_i -e_j: ij \in E(H)\}^\perp = 1$. \\

 In order to finish the proof of the claim, we need only show that our digraph $G$ is weakly connected. This is trivially true if $t = 0$, so we assume $t \geq 1$. Ignoring direction of edges, the operations we are allowed to use to get between vertices of $G$ (sets in $\binom{[m]}{t}$, that is) are the addition of $1$ and removal of some other element or the removal of $1$ and addition of some other element. Thus, each set containing $1$ is reachable from some set not containing $1$. If $S$ does not contain one and also does not contain some $i \neq 1$, we can first remove any $j$ from $S$ and add $1$, then remove $1$ and add $i$. This means $S - j + i$ is reachable from $S$. If there is no such $i$, then $S = \{2, \dots, m\}$. This implies the sets not containing $1$ are reachable from one another, so $G$ is weakly connected.
\end{proof}


\section{Proofs of local limit theorems}\label{sec:proofs}

\subsection{Preliminaries}

We use a few facts for the proof of \cref{thm:lattice_local_limit}. Throughout this section we assume $X$ is in isotropic position, i.e. $\E [XX^\dagger] = I_m$. This means $D_X = D$ and $B_X(\eps) = B(\eps)$.
\subsubsection{Fourier analysis}

\begin{defin}[Fourier transform]
If $Y$ is a random variable on $\R^m$, $\widehat{Y}: \R^m \to \C$ denotes the Fourier transform 
 $$\widehat{Y}(\theta) = \E[e^{2 \pi i \langle Y, \theta\rangle }].$$
\end{defin}

We will use the Fourier inversion formula, and our choice of domain will be the Voronoi cell in the dual lattice.
\begin{defin}[Voronoi cell]\label{defin:voronoi}

Define the Voronoi cell $D$ of the origin in $\cL^*$ to be the points as close to the origin as anything else in $\cL^*$, or 
$$D:= \{r \in \R^m: \|r\|_2 \leq \inf_{t \in \cL^* \setminus\{0\}} \|r - t\|_2\}.$$ 
Note that $\vol D = \det \cL^* = 1/\det \cL$, where $\det \cL$ is the volume of any domain whose translates under $\cL$ partition $\R^m$.
\end{defin}

\begin{fact}[Fourier inversion for lattices, \cite{LKP12}] For any random variable $Y$ taking values on a lattice $\cL$ (or even a lattice coset $v + \cL$), 
$$\Pr(Y = \lambda) = \det(\cL)\int_{D} \widehat{Y}(\theta) e^{-2\pi i \langle \lambda, \theta \rangle} d \theta. $$
for all $\lambda \in \cL$ (resp. $\lambda \in v + \cL$). Here $D$ is the Voronoi cell as in \cref{defin:voronoi}, but we could take $D$ to be any fundamental domain of $\cL$.\end{fact}

\subsubsection{Matrix concentration}
We use a special case of a result by Rudelson.
\begin{thm}[\cite{Ru99}]\label{thm:rud}
Suppose $X$ is an isotropic random vector in $\R^m$ such that $\|X\|_2\leq L$ almost surely. Let the $n$ columns of the matrix $M$ be drawn i.i.d from $X$. For some absolute constant $\refstepcounter{const} c_{{\theconst}}$ independent of $m,n$
$$\E\left\| \frac{1}{n}MM^\dagger - I_m\right\|_2 \leq c_{{\theconst}}L  \sqrt{\frac{\log n}{n}}.$$
In particular, there is a constant $\refstepcounter{const}\label{rud_const} c_{{\theconst}}$ such that with probability at least $1 - c_{{\theconst}}L  \sqrt{\frac{\log n}{n}}$ we have 
\begin{align*}&MM^\dagger \preceq 2I_m  \label{concentration}\tag{concentration}\\
\textrm{ and }&MM^\dagger \succeq \frac{1}{2}I_m \label{anticoncentration} \tag{anticoncentration}\\
\end{align*}
\end{thm}

\subsection{Dividing into three terms}

This section contains the plan for the proof of \cref{thm:lattice_local_limit}. The proof compares the Fourier transform of the random variable $My$ to that of a Gaussian; the integral to compute the difference of the Fourier transforms will be split up into three terms, which we will bound separately.

Let $M$ be a matrix whose columns $x_i $ are fixed vectors in $\cL$, and let $Y_M$ denote the random variable $My$ for $y$ chosen uniformly at random from $\{\pm 1/2\}^n$. This choice is made so that the random variable $Y_M$ takes values in the lattice coset $ \cL - \frac{1}{2}M \textbf{1}.$
Let $\Sigma_M$ be the covariance matrix of $Y_M$, which is given by 
$$ \Sigma_M = \frac{1}{4}\sum_{i = 1}^n x_i x_i^\dagger = \frac{1}{4}MM^\dagger.$$
Let $Y$ be a centered Gaussian with covariance matrix $\Sigma_M$. That is, $Y$ has the density 
$$G_M(\lambda) = \frac{1}{(2\pi)^{m/2} \sqrt{\det \Sigma_M}} e^{ - \frac{1}{2} \lambda^\dagger \Sigma_M^{-1} \lambda}.$$

 Observe that \cref{eq:thm_bound} in \cref{thm:lattice_local_limit} is equivalent to
$$|\prob(Y_M = \lambda) - \det(\cL) G_M(\lamdba)| \leq \frac{1}{(2\pi)^{m/2} \sqrt{\det \Sigma_M}} \cdot 2 m^2 L^2 n^{-1}$$ for $\lambda \in  \cL - \frac{1}{2}M \textbf{1}$. To accomplish this, we will show that $\widehat{Y_M}$ and $\widehat{Y}$ are very close. By Fourier inversion, for all $\lambda \in  \cL - \frac{1}{2}M \textbf{1}$,
\begin{eqnarray*}
|\prob(Y_M = \lambda) - \det(\cL) G_M(\lamdba)| =\\
  \det(\cL)\left|\int_{ D} \widehat{Y_M}(\theta) e^{-2\pi i \langle \lambda, \theta \rangle} d \theta - \int_{\R^m}\widehat{Y}(\theta) e^{-2\pi i \langle \lambda, \theta \rangle} d \theta \right|;
\end{eqnarray*}
recall the Voronoi cell $D$ from \cref{defin:voronoi}. Let $B(\eps) \subset \R^m$ denote the Euclidean ball of radius $\eps$ about the origin. If $B( \eps) \subset D$, then for all $\lambda \in \cL  - \frac{1}{2}M \textbf{1}$,

\begin{align}
|\prob(Y_M = \lambda) - \det(\cL) G_M(\lamdba)| \leq \nonumber \\
 = \det(\cL) \left(\underbrace{\int_{B(\eps)} |\widehat{Y_M}(\theta)- \widehat{Y}(\theta)| d\theta}_{ I_1} +    + \underbrace{\int_{\R^m \setminus B(\eps)} |\widehat{Y}(\theta)| d\theta }_{I_2} 
+ \underbrace{\int_{D \setminus B(\eps)} |\widehat{Y_M}(\theta)| d\theta}_{I_3}
 \right). \label{eq:the_integral}
\end{align}

We now show that this is decomposition holds for reasonably large $\eps$, i.e. $B(\eps) \subset D$.
\begin{lem}\label{isosplit} Suppose $ \eps \leq \frac{1}{2L}$. Then 
 Then $B( \eps) \subset D$; in particular, \cref{eq:the_integral} holds.

\end{lem}
\begin{proof}
Suppose $\theta \in B(\eps)$; we need to show that any nonzero element of the dual lattice has distance from $\theta$ at least $\eps$. It is enough to show that any such dual lattice element has norm at least $2 \eps$. Suppose $0 \neq \alpha \in \cL^*$. As $\supp(X)$ spans $\R^m$, for some $x \in \supp(X)$, we have $0 \neq \langle \alpha, x \rangle \in \Z$, so $\|x\|_2 \|\alpha\|_2 \geq |\langle \alpha, x \rangle | \geq 1;$ in particular $\|\alpha\|_2 \geq \frac{1}{L} \geq 2 \eps$.
\end{proof}

\subsubsection*{Proof plan}
We bound $I_1$ by using the Taylor expansion of $\widehat{Y_M}$ to see that, near the origin, $\widehat{Y_M}$ is very close to the unnormalized Gaussian $\widehat{Y}$. We bound $I_2$ using standard tail bounds for the Gaussian. The bounds for the first two terms hold for \emph{any} matrix $M$ satisfying \cref{concentration} and \cref{anticoncentration} and for the correct choice of $\eps$. Finally, we bound $I_3$ in \emph{expectation} over the choice of $M$. This is the only bound depending on the spanningness. 

\subsubsection{The term $I_1$: near the origin}
Here we show how to compare $\widehat{Y_M}$ to $\widehat{Y}$ near the origin in order to bound $I_1$ from \cref{eq:the_integral}.
The Fourier transform of the Gaussian $Y$ is $$\widehat{Y}(\theta) = \exp( - 2\pi^2 \theta^T \Sigma_M \theta). $$ There is a very simple formula for $\widehat{Y_M}$, the Fourier transform of $Y_M$. 

\begin{prop}\label{transform} If $M$ has columns $x_1, \dots x_n$, then 
\begin{equation}\widehat{Y}_M(\theta) = \prod_{j = 1}^n \cos({\pi \langle x_j, \theta \rangle}).\label{transform_eq}\end{equation}
\end{prop}
\begin{proof}
$\widehat{Y}_M(\theta) = \E_{y \in_R \{\pm 1/2\}^n}[e^{2\pi i \langle  \sum_{j =1}^n y_j x_j , \theta \rangle}] = \prod_{j =1}^n  \E_{y_j}[e^{2\pi i \langle  y_j x_j , \theta \rangle}] = \prod_{j = 1}^n \cos({\pi \langle x_j, \theta \rangle}).$
\end{proof} 

We can bound the first term by showing that near the origin, $\widehat{Y_M}$ is very close to a Gaussian. Recall that by Proposition \ref{transform}, 
$$ \widehat{Y_M}(\theta) = \prod_{j = 1}^n \cos({\pi \langle x_j, \theta \rangle}).$$
For $\theta$ near the origin, $\langle v_j, \theta \rangle$ will be very small. We will use the Taylor expansion of cosine near zero.
\begin{prop}\label{taylor}
For $x \in (-1/2,1/2)$, $\cos(\pi x) = \exp({\frac{\pi^2 x^2}{2}+ O(x^4)})$. 
\end{prop}
\begin{proof}
Let $\cos(\pi x) = 1 - y$ where $y \in [0,1)$. Then $\log(\cos(\pi x)) = \log(1 - y) = 1 - y + O(y^2)$. Since $\cos(\pi x) = 1 - \frac{\pi^2 x^2}{2} + O(x^4)$, we have that $y = \frac{\pi^2 x^2}{2} + O(x^4)$. Thus $\log(\cos(\pi x)) = 1 - \frac{\pi^2 x^2}{2} + O(x^4)$. The proposition follows.\end{proof}

We may now apply \cref{taylor} for $\|\theta\|_2$ small enough. 
\begin{lem}\label{isotaylor} Suppose $M$ satisfies \cref{concentration} and $\|\theta\|_2 < \frac{1}{2L}$. 
Then there exists a constant \refstepcounter{const}\label{tayc}$c_{\theconst}>0$ such that 
$$ \widehat{Y_M}(\theta)\leq \exp\left( - 2 \pi^2 \theta^\dagger\Sigma_M \theta + E\right)$$
for $|E| \leq c_{\theconst}  n L^2 \|\theta\|^4 $.

\end{lem}

\begin{proof} Because for all $i \in [n]$ we have $|\langle x_i, \theta \rangle| \leq \|x_i\|_2 \|\theta\|_2  < 1/2$, 
Proposition \ref{taylor} applies for all $i \in [n]$ and immediately yields that there is a constant $c$ such that 
$$\widehat{Y_M}(\theta) = \exp\left( - 2 \pi^2 \theta^T\Sigma_M \theta + E\right).$$
for $|E| \leq c\sum_{j = 1}^n \langle x_j, \theta\rangle^4$.
Next we bound the quartic part of $E$ by 
\begin{align*}
\sum_{j = 1}^n \langle x_j, \theta\rangle^4 &\leq \max_{j \in [n]} \| x_j\|_2^2 \|\theta\|_2^2 \sum_{j = 1}^n  \langle x_j, \theta\rangle^2\\
&\leq L^2  \|\theta\|_2^2 \theta^\dagger  \left( \sum_{j = 1}^n  x_j x_j^\dagger \right) \theta\\
&\leq 2  n L^2 \|\theta\|_2^4,
\end{align*}
and take $c_{\ref*{tayc}} = 2c$.
\end{proof}

\begin{lem}[First term]\label{i1}
Suppose $M$ satisfies \cref{anticoncentration} and \cref{concentration}. Further suppose that $L^2 n \eps^4 < 1$, and that $\eps < \frac{1}{2L}$. There exists \refstepcounter{const}\label{i1c}$c_{{\theconst}}$ with 
$$ I_1 \leq c_{{\theconst}} \frac{ m^2 L^2 n^{- 1}}{(2\pi)^{m/2}\sqrt{\det(\Sigma_M)}}.$$

\end{lem}

\begin{proof}
By \ref{concentration} and Lemma \ref{isotaylor}, 
\begin{eqnarray*}
I_1 = \int_{B(\eps)} |\widehat{Y_M}(\theta)- \widehat{Y}(\theta)| d\theta \leq \int_{B(\eps)} \widehat{Y}(\theta)\left|e^{c_{\ref*{tayc}} L^2 n \|\theta\|_2^4 } - 1\right| d \theta.
\end{eqnarray*}
Let the constant $c$ be such that $|e^{c_{\ref*{tayc}} x} - 1| \leq c|x|$ for $x \in [-1, 1]$.  Thus 
$$
I_1 \leq  c L^2 n \int_{B(\eps)} \widehat{Y}(\theta) \|\theta\|_2^4 d \theta. 
$$
By \cref{anticoncentration},
\begin{equation}
I_1 \leq c  L^2 n^{-1} \int_{B(\eps)} \widehat{Y}(\theta) \left( \theta^\dagger \Sigma_M \theta \right)^2 d \theta . \label{i1bd}
\end{equation}

Note that $(2\pi)^{m/2}\sqrt{\det(\Sigma_M)}\widehat{Y}$ is equal to the density of $W = \frac{1}{2\pi}\Sigma_M^{-1/2}G$, where $G$ is a Gaussian vector with identity covariance matrix. $\Sigma_M^{-1/2}$ exists because \cref{anticoncentration} holds. Further,
$W^\dagger \Sigma_M W = \frac{1}{4\pi^2 }\|G\|_2^2.$  Therefore 
\begin{align*}
 \int_{\R^m} \widehat{Y}(\theta) \left( \theta^\dagger \Sigma_M \theta \right)^2 d \theta &= \frac{1}{(2\pi)^{m/2}\sqrt{\det(\Sigma_M)}}\E_W\left[\left(W^\dagger \Sigma_M W \right)^2\right]\\
  &= \frac{1}{16\pi^4(2\pi)^{m/2}\sqrt{\det(\Sigma_M)}}\E_G\left[ \|G\|_2^4\right]\\
& =  \frac{1}{16\pi^4(2\pi)^{m/2} \sqrt{\det(\Sigma_M)}}(2m + m^2)\\
&\leq \frac{3  m^2}{16\pi^4(2\pi)^{m/2}\sqrt{\det(\Sigma_M)}}.\end{align*}
Plugging this into \eqref{i1bd} and setting $c_{\ref*{i1c}} = \frac{3}{ \pi^4} c$ completes the proof.

\end{proof}


\subsubsection{The term $I_2$: Bounding Gaussian mass far from the origin}
Here we bound the term $I_2$ of \cref{eq:the_integral}, which is not too difficult.
\begin{lem}[Third term]\label{i3}Suppose $M$ satisfies \cref{anticoncentration} holds and that $\eps^2 \geq \frac{16m}{\pi^2n}$. Then 
$$ I_2 \leq \frac{e^{-\frac{\pi^2}{8} \eps^2n}}{(2\pi)^{m/2}\sqrt{\det(\Sigma_M)}}.$$
\end{lem}
\begin{proof} 
If $M$ satisfies \cref{anticoncentration}, then $B(\eps) \supset \frac{1}{2} \{\theta: \theta^\dagger \Sigma_M \theta \geq n \eps\}:=B_M(\eps/2)$. If we integrate over $B_M(\eps/2)$ and change variables, it remains only to calculate how much mass of a standard normal distribution is outside a ball of radius larger than the average norm. From, say, Lemma 4.14 of \cite{LKP12} , if $\eps^2 \geq \frac{16m}{\pi^2n}$ then 
$$\int_{\R^m \setminus B_M(\eps/2)} |\widehat{Y}(\theta)| d\theta \leq \frac{e^{-\frac{\pi^2}{8} \eps^2n}}{(2\pi)^{m/2}\sqrt{\det(\Sigma_M)}}.$$

\end{proof}



\subsubsection{The term $I_3$: Bounding the Fourier transform far from the origin}
It remains only to bound the term $I_3$ of \cref{eq:the_integral} which is given by
$$I_3 = \int_{D \setminus B(\eps)} |\widehat{Y_M}(\theta)| d\theta.$$
This is the only part in which spanningness plays a role. If $\eps$ is at most the spanningness (see \cref{defin:spanningness}), we can show $I_3$ is very small with high probability by bounding it in expectation over the choice of $M$. The proof is a simple application of Fubini's theorem.

 \begin{lem}\label{i2} 
If $\E XX^\dagger = I_m$ and $\eps \leq s(X)$, then 
$$\E[ I_3 ] \leq \det(\cL^*) e^{-2 \eps^2 n} $$

\end{lem}

\begin{proof}
By Fubini's theorem,
\begin{align} \E_M[ I_3 ]&= \int_{D \setminus B(\eps)} \E|\widehat{Y_M}(\theta)| d\theta \nonumber\\
& \leq \det(\cL^*)\sup\{ \E[|\widehat{Y_M}(\theta)|]:\theta\in D \setminus B(\eps) \}.\label{eq:fubini}
\end{align} 
By Proposition \ref{transform} and the independence of the columns of $n$, 
$$ \E_M[|\widehat{Y_M}(\theta)|] = \left( \E|\cos({\pi \langle X, \theta \rangle})|\right)^n.$$
Thus, 
\begin{align}\sup\{ \E[|\widehat{Y_M}(\theta)|]:\theta\in D \setminus B(\eps) \} \leq \left(\sup\{ \E[|\cos({\pi \langle X, \theta \rangle})|]:\theta\in D \setminus B(\eps) \}\right)^n.\label{eq:sup}\end{align}
$|\cos(\pi x)|$ is periodic with period $1/2$, so it is enough to consider $\langle X, \theta \rangle \bmod{1}$, where $x \bmod{1}$ is taken to be in $[-1/2, 1/2)$. Note that for $|x| \leq 1/2$, $|\cos(\pi x)| = \cos(\pi(x)) \leq 1 - 4x^2$, so
$$ \E[|\cos({\pi \langle X, \theta \rangle})|] \leq 1 - 4\E[(\langle X, \theta \rangle \bmod{1})^2] = 1 - 4\tilde{X}(\theta)^2$$
By the definition of spanningness and the assumption in the hypothesis that $\eps \leq s(X)$, we know that every vector with $\tilde{X}(\theta) \leq d(\theta, \cL^*)/2 = \|\theta\|/2$ is either in $\cL^*$ or has $\tilde{X}(\theta) \geq \eps$. Thus, for all $\theta \in D$, $\tilde{X}(\theta) \geq \max \{\|\theta\|/2,\eps\}$, which is at least $\eps/2$ for $\theta \in D \setminus B(\eps)$. Combining this with \cref{eq:sup} and using $1 - x \leq e^{-x}$ implies 
$$\sup\{ \E[|\widehat{Y_M}(\theta)|]:\theta\in D \setminus B(\eps) \} \leq   e^{-2 \eps^2 n}.$$
Plugging this into \cref{eq:fubini} completes the proof.
\end{proof}

\subsection{Combining the terms}

Finally, we can combine each of the bounds to prove \cref{thm:lattice_local_limit}.
\begin{proof}[Proof of \cref{thm:lattice_local_limit}]
Recall the strategy: we have some conditions (the hypotheses of Lemma \ref{isosplit}) under which we can write the difference between the two probabilities of interest as a sum of three terms, and we have bounds for each of the terms (Lemma \ref{i1}, Lemma \ref{i2}, and Lemma \ref{i3}) respectively. Our expression depends on $\eps$, and so we must choose $\eps$ satisfying the hypotheses of those lemmas. These are as follows:
\begin{enumerate}[label = (\roman*)]
\item\label{isosplitconst} To apply \ref{isosplit} we need $\eps \leq \frac{1}{2L},$
\item\label{i1const} for Lemma \ref{i1} we need $L^2 n\eps^4  \leq 1$, 
\item\label{i3const} to apply Lemma \ref{i3}, we need $\eps^2 \geq \frac{16m}{\pi^2n}$, and 
\item\label{i2const} for Lemma \ref{i2} we need $\eps \leq s(X)$.

\end{enumerate}
It is not hard to check that setting
 $$\eps = L^{-1/2}n^{-1/4}$$
 will satisfy the four constraints provided $n\geq 16 L^2$,  $n \geq (16 m L)^2/\pi^4$, and $n \geq s(X)^{-4} L^{-2}$. However, the first condition follows from the second because $L \geq \sqrt{m}$ (this follows from $\E XX^\dagger = I_m$, which implies $\E[\|X\|_2^2] = m$), so 
 $$n \geq (16 m L)^2/\pi^4 \textrm{ and } n \geq s(X)^{-4} L^{-2}$$
 suffice. By \ref{isosplit} we have \begin{align*}
|\prob(Y_M = \lambda) - \det(\cL) G_M(\lamdba)| \leq \nonumber \\
 = \det(\cL) \left(\underbrace{\int_{B(\eps)} |\widehat{Y_M}(\theta)- \widehat{Y}(\theta)| d\theta}_{ I_1} +    + \underbrace{\int_{\R^m \setminus B(\eps)} |\widehat{Y}(\theta)| d\theta }_{I_2} 
+ \underbrace{\int_{D \setminus B(\eps)} |\widehat{Y_M}(\theta)| d\theta}_{I_3}
 \right). 
\end{align*}

 By \cref{i2} and Markov's inequality, $I_3$ is at most $ e^{-\eps^2 n}$ with probability at least $1 - e^{-\eps^2 n} \det(\cL^*)$. By \cref{thm:rud}, \cref{anticoncentration} and \cref{concentration} hold for $M$ with probability at least $1 - c_{\ref*{rud_const}} L\sqrt{( \log n)/n}$. If $n$ is at least a large enough constant times $L^2 \log^2 \det \cL$, $e^{-\eps^2 n} \det(\cL^*)$ is at most $L\sqrt{( \log n)/n}$. Thus, all three events hold with probability at least \refstepcounter{const}\label{failconst}$1 - c_{\ref*{failconst}} L\sqrt{( \log n)/n}$ over the choice of $M$. Condition on these three events, and plug in the bounds given by \cref{i1} and \cref{i3} for $I_1$ and $I_2$ and the bound $e^{-\eps^2 n} = e^{- \sqrt{n}/L }$ for $I_3$ to obtain the following:

\begin{align}
|\prob(Y_M = \lambda) - \det(\cL) G_M(\lamdba)| \nonumber \\
 \leq \det(\cL) \left(\frac{ m^2 L^2 n^{- 1}}{(2\pi)^{m/2}\sqrt{\det(\Sigma_M)}} 
 +   \frac{e^{-\frac{\pi^2}{8} \sqrt{n}/L}}{(2\pi)^{m/2}\sqrt{\det(\Sigma_M)}} +    e^{- \sqrt{n}/L }\right). \nonumber\\
\leq \frac{\det(\cL)}{(2\pi)^{m/2}\sqrt{\det(\Sigma_M)}} \left(m^2 L^2 n^{- 1}  +   e^{-\frac{\pi^2}{8} \sqrt{n}/L} +   (2\pi)^{m/2}\sqrt{\det(\Sigma_M)} e^{- \sqrt{n}/L }\right)\nonumber\\
 \leq \frac{\det(\cL)}{(2\pi)^{m/2}\sqrt{\det(\Sigma_M)}} \left(m^2 L^2 n^{- 1} +  2 e^{\frac{m}{2} \log (4\pi n) - \sqrt{n}/L } 
 \right),
  \label{eq:the_bound}
  \end{align}
where the last inequality is by \cref{concentration}. If $\refstepcounter{const}\label{nlower} c_{{\theconst}}$ is large enough, the quantity in parentheses in \cref{eq:the_bound} is at most $2m^2 L^2/n$ and the combined failure probability of the three required events is at most $c_{\ref*{failconst}} L\sqrt{\frac{\log n}{n}}$ provided
\begin{align}
n \geq N_0 = c_{{\theconst}} \max\left\{m^2 L^2( \log m + \log L)^2, s(X)^{-4}L^{-2},L^2 \log^2 \det \cL\right\}.
\end{align}

\end{proof}

\subsection{Weaker moment assumptions}\label{sec:moments}
We now sketch how to extend the proof of \cref{thm:lattice_local_limit} to the case $(\E \|X\|_2^\eta)^{1/\eta} = L< \infty$ for some $\eta > 2$, weakening the assumption that $\supp X \subset B(L)$. 

\begin{thm}[lattice local limit theorem for $>2$ moments]\label{thm:moment_local_limit} Let $X$ be a random variable on a lattice $\cL$ such that $\E XX^\dagger = I_m$, $(\E \|X\|_2^\eta)^{1/\eta} = L< \infty$ for some $\eta > 2$, and $\cL = \spn_\Z \supp X$. Let $G_M$ be as in \cref{thm:lattice_local_limit}.
There exists 
\begin{align}N_2 = \poly(m, s(X), L, \frac{1}{\eta - 2},\log \left(\det \cL\right))^{1 + \frac{1}{\eta - 2}}\label{eq:nmoment}
\end{align}
such that for $n \geq N_2$, with probability at least $1 - 3n^{- \frac{\eta-2}{2 + \eta}}$ over the choice of columns of $M$, the following two properties of $M$ hold:
 \begin{enumerate}
 \item $MM^\dagger \succeq \frac{1}{2}n I_m$; that is, $MM^\dagger - \frac{1}{2}n I_m$ is positive-semidefinite.
\item For all $\lambda \in \cL -  \frac{1}{2} M \textbf 1$,
\begin{align}
\left|\Pr_{y_i \in \{\pm 1/2\}} [M\mathbf y = \mathbf\lambda] - G_{M}(\lambda)\right| \leq  G_M(0)\cdot 2m^2 L^2 n^{- \frac{\eta-2}{2 + \eta}}. \label{eq:thm_bound}
\end{align}
\end{enumerate}
\end{thm}
Before proving the theorem, note that it allows us to extend our discrepancy result to this regime. The proof of the next corollary from \cref{thm:moment_local_limit} is identical to the proof of \cref{thm:lattice_disc} from \cref{thm:lattice_local_limit}.

\begin{cor}[discrepancy for $>2$ moments]\label{cor:moment_disc}Suppose $X$ is a random variable on a nondegenerate lattice $\cL$. Suppose $\Sigma:=\E [XX^\dagger]$ has least eigenvalue $\sigma$, $(\E \|Z\|_2^\eta)^{1/\eta} = L< \infty$ for some $\eta >2$ where $Z:=\Sigma^{-1/2}X$, and that $\cL=\spn_\Z \supp X$. If $n \geq N_3$ then
$$ \disc_*( M) \leq2 \rho_*(\cL)$$
with probability at least $1 - 3n^{- \frac{\eta-2}{2 + \eta}}$, where \refstepcounter{const}\label{ndisc} 
\begin{align}N_3 =c_{\ref*{ndisc}} \max\left\{\frac{R^2_* \rho_*(\cL)^2}{\sigma}, N_2\left(m, s(\Sigma^{-1/2}X), L, \frac{\det \cL}{\sqrt{\det \Sigma}}\right)\right\} \label{eq:ndisc}\end{align}
 for $N_0$ as in \cref{eq:nlim}.

\end{cor}

\begin{proof}[Proof sketch of \cref{thm:moment_local_limit}]
We review each step of the proof of \cref{thm:lattice_local_limit} and show how it needs to be modified to accomodate the weaker assumptions. Recall that, to prove \cref{thm:lattice_local_limit}, we had some conditions (the hypotheses of Lemma \ref{isosplit}) under which we can write the difference between the two probabilities of interest as a sum of three terms, and we have bounds for each of the terms (Lemma \ref{i1}, Lemma \ref{i2}, and Lemma \ref{i3}) respectively. We also need an analogue of \cref{thm:rud} which tells us that \cref{anticoncentration} and \cref{concentration} hold with high probability. Neither Lemma \ref{i2} nor Lemma \ref{i3} use bounds on the moments of $\|X\|_2$, so they hold as-is. Let's see how the remaining lemmas must be modified:
\begin{description}
\item[Matrix concentration:] By Theorem 1.1 in \cite{SV13}, \cref{concentration} and \cref{anticoncentration} hold with probability at least $1 - n^{-\frac{\eta-2}{2 + \eta}}$  provided $n \geq \poly(m, L, \frac{1}{\eta-2})^{1 + \frac{1}{\eta - 2}}$. 
\item[Lemma \ref{isosplit}:]The bound $\eps \leq \frac{1}{2L}$ becomes $ \eps \leq \frac{1}{4} L^{-\frac{\eta}{\eta-2}}$. To prove Lemma \ref{isosplit} it was enough to show $\alpha$ was at least twice the desired bound on $\eps$ for $\alpha \neq 0 \in \cL^*$. Here we do the same, but to show $\alpha$ is large we consider the random variable $Y \geq 1$ defined by conditioning $|\langle \alpha, X \rangle|$ on $\langle \alpha, X \rangle \neq 0$. Recall that we assume $X$ is isotropic. Let $\Pr[ \langle \alpha, X \rangle \neq 0]$, so that $\|\alpha\|^2 = p \E[Y^2]$ and
$ L\|\alpha\| \geq (\E |\langle \alpha, X \rangle|^{\eta})^{\frac{1}{\eta}} 
= p^{\frac{1}{ \eta}}(\E[ Y^{\eta}])^{\frac{1}{\eta}}
\geq p^{\frac{1}{ \eta}} (\E[Y^{2}])^{\frac{1}{2}}$ by H\"older's inequality. Cancelling $p$ from the two inequalities and using $Y \geq 1$ yields the desired bound.
\item[Lemma \ref{i1}:] The analogue of this lemma will require $L^2 n^{1 + \frac{4 }{2 + \eta}}\eps^4 < 1$ and $\eps < \frac{1}{4L}n^{ - \frac{2}{2 + \eta} }$, and will hold with probability at least $1 - n^{- \frac{\eta}{4 + \eta}}$ over the choice of columns of $M$. The numerator of the right-hand side becomes $m^2 L^2 n^{- \frac{\eta-2}{2 + \eta}}$. \cref{i1} followed from \cref{isotaylor}. Here the analogue of \cref{isotaylor} holds with $|E| \leq c_{\ref*{tayc}}  n^{1 + \frac{4 }{2 + \eta}} L^2 \|\theta\|^4$ if $\|z_i\| \leq L n^{\frac{2}{2 + \eta} }  \textrm{ for all }i \in [n]$, which holds with probability $1 - n^{- \frac{\eta-2}{2 + \eta}}$ by Markov's inequality. The rest of the proof proceeds the same.
\end{description}
The new constraints on $\eps$ will be satisfied if we take $$ \eps = n^{ - \frac{4 + \eta}{12 + 3\eta} },$$
and $n \geq \max \left\{(4L)^{\frac{12 + 6\eta}{\eta-2}}, \frac{16}{\pi^2} m^{\frac{6 + 3\eta}{2\eta-4}}\right\}.$ The rest of the proof proceeds as for \cref{thm:lattice_local_limit}. \end{proof}


\section{Random unit columns}\label{sec:unit}


Let $X$ be a uniformly random element of the sphere $\mathbb{S}^{m-1}$. Again, let $M$ be an $m \times n$ matrix with columns drawn independently from $X$. Note that $X$ is not a lattice random variable. This time $\Sigma = \frac{1}{m} I_m,$ and $\|\Sigma^{-1/2} X\|_2$ is always at most $m$.

We are essentially going to prove a local limit theorem, only this time we will not precisely control the probability of hitting a point but rather the expectation of a particular function. The function will essentially be the indicator of the cube, but it will be modified a bit to make it easier to handle. Let $B$ be the function, which we will determine later. Recall that, once $M$ is chosen, $Y_M$ is the random variable obtained by summing the columns of $M$ with i.i.d $\pm 1/2$ coefficients. $\Sigma_M$ is $MM^\dagger/4$, and $Y$ is the Gaussian with covariance matrix $\Sigma_M$
We will try to show that, with high probability over the choice of $M$, 
$\E B(Y_M) \sim \E B(Y)$. If $B$ is supported only in $[-K, K]^m$, to show that $\disc M < K$ it suffices to show that 
$$ |\E B(Y_M) -  \E B(Y)| < \E B(Y).$$

\subsection{Nonlattice likely local limit}

We now investigate a different extreme case in which $\spn_\Z \supp X$ is dense in $\R^m$. In this case the ``dual lattice" is $\{0\}$, so we define pseudodual vectors to be those vectors with $\tilde{X}(\theta) \leq \|\theta\|_X/2$, and the spanningness to be the least value of $\tilde{X}(\theta)$ at a nonzero pseudodual vector.

\begin{thm}\label{thm:nonlattice_local_limit} Suppose $\E XX^\dagger = I_m$, $\supp X \subset B(L)$, and that $s(X)$ is positive. Let $B:\R^m \to \R$ be a nonnegative function with $\|B\|_1 \leq 1$ and $\|\widehat{B}\|_1 \leq \infty$. If 
$$ n \geq N_1 = c_{\ref*{unitnlower}} \max\left\{m^2 L^2( \log M + \log L)^2, s(X)^{-4}L^{-2},L^2 \log^2 \|B\|_1\right\},$$
then with probability at least $ c_{\ref*{failconst}} L\sqrt{\frac{\log n}{n}}$ over the choice of $M$ we have 
$$|\E[ B(Y_M)] - \E[B(Y)]| \leq 2 m^2 L^2 n^{-1}$$
and $MM^\dagger \succeq \frac{1}{2}n I_m.$

\end{thm}

\begin{proof} 
By Plancherel's theorem, 
$$ \E[ B(Y_M)] - \E[B(Y)] = \int_{\R^m} \widehat{B}(\theta) (\widehat{Y_M}(\theta)  - \widehat{Y}(\theta)) d \theta.$$
Again, we can split this into three terms:
\begin{align}
\left| \int_{\R^m}  \widehat{B}(\theta)(\widehat{Y_M}(\theta) - \widehat{Y}(\theta)) d\theta \right|\leq \nonumber\\ 
\underbrace{\int_{B(\eps)} | \widehat{B}(\theta)| |\widehat{Y_M}(\theta) - \widehat{Y}(\theta) | d \theta}_{J_1}  
 + \underbrace{\int_{R^m \setminus B(\eps)} | \widehat{B}(\theta)\widehat{Y_M}(\theta)|d\theta}_{J_2} 
 + \underbrace{\int_{R^m \setminus B(\eps)} | \widehat{B}(\theta)\widehat{Y}(\theta)|d\theta.}_{J_3} 
\end{align}
The proofs of the next two lemmas are identical to that of \cref{i1} and \cref{i3}, respectively, except one uses the assumption $\|B\|_1 \leq 1$, which implies $\|\widehat{B}\|_\infty \leq 1$, to remove $\widehat{B}$ from the integrand.
\begin{lem}[First term]\label{j1}
Suppose \cref{anticoncentration} and \cref{concentration} hold. Further suppose that $L^2 n \eps^4 < 1$, $\eps < \frac{1}{2L}$, and that $\|B\|_1 \leq 1$. There exists \refstepcounter{const}\label{j1c}$c_{{\theconst}}$ with 
$$ J_1 \leq c_{{\theconst}} \frac{m^2 L^2 n^{- 1}}{(2\pi)^{m/2}\sqrt{\det(\Sigma_M)}}.$$
\end{lem}
\begin{lem}[Third term]\label{j3}Suppose \cref{anticoncentration} holds, $\eps^2 \geq \frac{16m}{\pi^2n}$, and $\|B\|_1 \leq 1$. Then 
$$ J_3 \leq \frac{e^{-\frac{\pi^2}{8} \eps^2n}}{(2\pi)^{m/2}\sqrt{\det(\Sigma_M)}}.$$
\end{lem}
The proof of the next lemma is the same as that of \cref{i3}, except in the derivation of \cref{eq:fubini} instead of integrating over $D$ one must integrate over the whole of $\R^m\setminus B(\eps)$ against $\widehat{B}$, hence $\det L^*$ is replaced by $\|\widehat{B}\|_1$. 
 \begin{lem}\label{j2} 
 If $X$ is in isotropic position and $\eps \leq s(X)$, then 
$$\E[ J_2] \leq \|\widehat{B}\|_1 e^{-2 \eps^2 n} $$
\end{lem}
We now proceed to combine the termwise bounds. As before, we may choose $\eps = n^{1/2} L^{-1/2}$ provided  
$$n \geq (16 m L)^2/\pi^4 \textrm{ and } n \geq s(X)^{-4} L^{-2},$$
and with probability at least $1- \|\widehat{B}\|_1 e^{-\eps^2 n} - c_{\ref*{rud_const}} L \sqrt{ \frac{\log n}{n}},$ we have $J_2$ at most $e^{-\eps^2 n}$ and \cref{concentration}, \cref{anticoncentration} hold. Condition on these events. As in the proof of $\cref{thm:lattice_local_limit}$, we have 
\begin{align}
\left| \int_{\R^m}  \widehat{B}(\theta)(\widehat{Y_M}(\theta) - \widehat{Y}(\theta)) d\theta \right|\leq \frac{1}{(2\pi)^{m/2}\sqrt{\det(\Sigma_M)}} \left(m^2 L^2 n^{- 1} +  2 e^{\frac{m}{2} \log (4\pi n) - \sqrt{n}/L } 
 \right).
  \label{eq:unit_bound}
  \end{align}
If $\refstepcounter{const}\label{unitnlower} c_{{\theconst}}$ is large enough, the quantity in parentheses in \cref{eq:the_bound} is at most $2m^2 L^2/n$ and the combined failure probability of all the required events is at most $c_{\ref*{failconst}} L\sqrt{\frac{\log n}{n}}$ provided
\begin{align}
n \geq N_1 = c_{{\theconst}} \max\left\{m^2 L^2( \log M + \log L)^2, s(X)^{-4}L^{-2},L^2 \log^2 \|B\|_1\right\}.
\end{align}
\end{proof}


\subsection{Discrepancy for random unit columns}

\begin{lem}\label{uniti2ev}Let $X$ be a random unit vector. Then
$$ s( X) \geq \refstepcounter{const}\label{unitev}c_{{\theconst}}.$$
for some fixed constant $c_{{\theconst}}$.
\end{lem}
Before we prove the lemma, let's show how to use it and \cref{thm:nonlattice_local_limit} to prove discrepancy bounds.

\begin{proof}[Proof of \cref{thm:unit_disc}] Let $X$ be a random unit vector. We need to choose our function $B$. 
\begin{defin} For $K > 0$, let $B = \frac{1}{(2K)^{2m}} 1_{[-K,K]^m} * 1_{[-K,K]^m}$. Alternately, one can think of $B$ as the density of a sum of two random vectors from the cube $[-K,K]^m$. \end{defin}
It's not hard to show $\|B\|_1 = 1$ using that $B$ is a density that and that $\|\widehat{B}\|_1 = \frac{1}{(2K)^{m}}$ using Plancherel's theorem. Next, we apply \cref{thm:nonlattice_local_limit} to $Z = \sqrt{m} X$; in order to apply the theorem we need $$ n \geq N_2 := \refstepcounter{const}\label{unitdiscrep} c_{{\theconst}} \max\left\{m^3 \log^2 m, m^{-1}, m^3 \log^2 (1/K)\right\}.$$ Thus, we may take 
$$n \geq c_{\ref*{cubeconst}} m^3 \log^2m \textrm{ and } K = c_{\ref*{expconst}} e^{-\sqrt{\frac{n}{m^3}}}. $$ 
We also need to obtain a lower bound on $\E [B(Y)]$ in order to use the bound on $|\E[B(Y_M)] - \E[B(Y)]|$ to deduce that $\E[B(Y_M)] > 0$, or equivalently that $\disc M \leq 2K$. The quantity $\E B(Y)$ is at least the least density of $Y$ on the support of $B$. The support of $B$ is contained within a $2K\sqrt{m}$ Euclidean ball. Using the property $MM^\dagger \geq \frac{1}{2}n  I_m$, the density of $Y$ takes value at least 
\begin{align*} \frac{1}{(2\pi)^{m/2}\sqrt{\det(\Sigma_M)}} e^{-2  \sigma_{min}(MM^\dagger)^{-1} 4 K^2 m} &\geq 
\frac{1}{(2\pi)^{m/2}\sqrt{\det(\Sigma_M)}} e^{-16 K^2 m/n}.
\end{align*}
Since the error term in \cref{thm:nonlattice_local_limit} is at most $\frac{1}{(2\pi)^{m/2}\sqrt{\det(\Sigma_M)}}2 m^2 L^2 n^{-1}$, to deduce $\disc M \leq K$ it is enough to show  $ e^{-16 K^2 m/n} > 2m^3 n^{-1}$; indeed this is true with $K = c_{\ref*{expconst}} e^{-\sqrt{\frac{n}{m^3}}}$ and $n \geq m^3 \log^2 m$ for suitably large $c_{\ref*{expconst}}$.\end{proof}

\subsubsection{Spanningness for random unit vectors}
We now lower bound the spanningness for random unit vectors. We use the fact that for large $m$ the distribution of a random unit vector behaves much like a Gaussian upon projection to a one-dimensional subspace. 
\begin{proof}[Proof of \cref{uniti2ev}]
Let $\|\theta\|_X =\frac{1}{\sqrt{m}}\delta > 0$. We split into two cases. In the first, we show that if $\delta = O(\sqrt{m})$, then $\theta$ is not pseudodual. In the second, we show that if $\delta = \Omega(\sqrt{m})$ then $\tilde{X}(\theta)$ is at least a fixed constant. This establishes that $s(X)$ is at least some constant. 

By rotational symmetry we may assume $\theta$ is $\delta e_1$, where $e_1$ is the first standard basis vector, so $$\tilde{X}(\theta)^2 =  \E[(\langle X, \theta \rangle \bmod{1})^2 ] = \E[(\delta X_1 \bmod{1})^2 ].$$ 

 
We now try to show $\theta$ is not a pseudodual vector if $\delta = O( \sqrt{m})$. Recall that $X$ is a random unit vector; it is easier to consider the density of $X_1$. The probability density function of $\delta X_1$ for $x < \delta$ is proportional to $(1-(x/\delta)^2)^{\frac{m-3}{2}} =: f_\delta(x)$.
Integrating this density gives us 
\begin{align*}
\int_{-\delta}^\delta \left(1-\left(\frac{x}{\delta}\right)^2\right)^{\frac{m-3}{2}}dx &= \delta\int_{0}^\delta (1 - x)^{\frac{m-3}{2}} x^{-\frac{1}{2}} dx\\
 &= \frac{\delta \Gamma\left( \frac{m-1}{2}\right) \Gamma\left( \frac{1}{2}\right)}{ \Gamma\left( \frac{m}{2}\right)}\\
 &= \frac{\delta \sqrt{\pi}}{ \sqrt{m}}\left(1 + o(1)\right).
\end{align*}
Therefore, we obtain 
\begin{align*}
\E[( \delta X_1 \bmod{1})^2 ] &=  \frac{ \Gamma\left( \frac{m}{2}\right)}{\delta \Gamma\left( \frac{m-1}{2}\right) \Gamma\left( \frac{1}{2}\right)} \int_{-\delta}^{\delta} (x \bmod 1)^2 \left(1-\left(\frac{x}{\delta}\right)^2\right)^{\frac{m-3}{2}}dx. \\
\end{align*}
If we simply give up on all the $x$ for which $|x| > 1/2$, we obtain the following lower bound for the above quantity:
\begin{align*}
&\frac{ \Gamma\left( \frac{m}{2}\right)}{\delta \Gamma\left( \frac{m-1}{2}\right) \Gamma\left( \frac{1}{2}\right)}\left[\int_{-\delta}^{\delta} x^2 \left(1-\left(\frac{x}{\delta}\right)^2\right)^{\frac{m-3}{2}}dx - 2\int_{1/2}^{\delta} x^2 \left(1-\left(\frac{x}{\delta}\right)^2\right)^{\frac{m-3}{2}}dx \right]\\
&= \frac{\delta^2}{m} - \frac{2 \Gamma\left( \frac{m}{2}\right)}{\delta \Gamma\left( \frac{m-1}{2}\right) \Gamma\left( \frac{1}{2}\right)} \int_{1/2}^{\delta} x^2 \left(1-\left(\frac{x}{\delta}\right)^2\right)^{\frac{m-3}{2}}dx\\
&\geq \frac{\delta^2}{m} -  (1 + o(1))\frac{2\sqrt{m - 3}}{ \delta \sqrt{ \pi}} \int_{1/2}^{\infty} x^2 e^{-\frac{(m-3)x^2}{2 \delta^2}}dx.\\
& = \frac{\delta^2}{m} -  (1 + o(1))  \frac{2^{3/2}\delta^2 }{ m} \left(\frac{1}{\sqrt{2\pi}}\int_{\frac{\sqrt{m - 3}}{2 \delta}}^{\infty} u^2 e^{-\frac{u^2}{2}}du\right)
\end{align*}
The integral in parentheses is simply the contribution to the variance of the tail of a standard gaussian, and can be made an arbitrarily small constant by making $\delta/\sqrt{m}$ small. Thus, for $\delta$ at most $\delta \leq  \refstepcounter{const}\label{mlb1}
c_{{\theconst}}\sqrt{m}$, 
the last line above expression is at least $.6^2\frac{\delta^2}{m} = .6^2\| \theta\|_X^2$, so $\theta$ is not pseudodual.\\

Next we must handle $\delta$ larger than $c_{{\theconst}}\sqrt{m}$; we will show that in this case $\tilde{X}(\theta)$ is at least some fixed constant. We use the fact that $f_\delta$ is unimodal, so for any $k \neq 0$,  $\int_{k - 1/2}^{k + 1/2} (x \bmod 1)^2  f_\delta(x) dx$ is at least the mass of $f_\delta(x)$ between $k-1/2$ and $k$ times the integral of $(x \bmod 1)^2$ on this region (that is, $1/24$). This product is then at least which is at least one 48th of the mass of $f_\delta(x)$ between $k-1/2$ and $k + 1/2$. Taken together, we see that 
\begin{equation}\label{leakage}\E[(\delta X_1 \bmod 1)^2] \geq \frac{1}{48} \Pr[|\delta X_1| \geq 1/2].\end{equation} 
We will lower bound the left-hand side by a small constant for $\delta = \Omega(\sqrt{m})$. We can do so by bounding the ratio of $\int_{-1/2}^{1/2} f_\delta(x)$ to $\int_{1/2}^\infty f_\delta(x)$. To this end we will translate and scale the function 
$$g_\delta(x) = 
\left\{\begin{array}{cc}
f_\delta(x) & x \geq 1/2\\
0 & x < 1/2
\end{array}\right.
$$
to dominate $f_\delta(x)$ for $x \in [0, 1/2]$. Let us find the smallest scaling $a > 0$ such that $a g_\delta (x + 1/2) \geq f_\delta(x)$ for $x \in [0,1/2]$; equivalently,
$a f_\delta (x + 1/2) \geq f_\delta(x)$ for $x \in [0,1/2]$. If we find such an $a$, we'll have $\int_0^{1/2} {f_\delta(x)} \leq a \int_0^\infty {g_\delta(x)} dx$, or $\Pr[x \in [-1/2, 1/2]] \leq a (1 - \Pr[x \in [-1/2, 1/2]]).$
We need
\begin{align*}
a (1-((x+1/2)/\delta)^2)^{\frac{m-3}{2}} \geq (1-(x/\delta)^2)^{\frac{m-3}{2}}, 
\end{align*}
or 
\begin{align*}a=\max_{x \in [0,1/2]} \left(\frac{1-(x/\delta)^2}{1-((x+1/2)/\delta)^2}\right)^{\frac{m-3}{2}}\\
= \max_{x \in [0,1/2]} \left(\frac{\delta^2-x^2}{\delta^2-(x+1/2)^2}\right)^{\frac{m-3}{2}}\\
\leq \left(\frac{\delta^2}{\delta^2-1}\right)^{\frac{m-3}{2}} = \left(1-\frac{1}{\delta^2} \right)^{-\frac{m-3}{2}}\\
\leq e^{\frac{(m-3)}{2\delta^2}}.
\end{align*}
As discussed, we now have $\Pr[x \in [-1/2, 1/2]] \leq a (1 - \Pr[x \in [-1/2, 1/2]]).$ Equivalently, $\Pr[x \in [-1/2, 1/2]] \leq a/(1 + a)$. Therefore
 $$\Pr[|x| > 1/2] \geq 1/(1 + a) \geq .5 e^{-\frac{m-3}{2\delta^2}}.$$
If $\delta \geq c_{\ref*{mlb1}} \sqrt{m}$, this and \cref{leakage} imply $\E[(\delta X_1 \bmod 1)^2] = \E[(\langle \theta, X \rangle \bmod 1)^2$ is at least some constant. Thus, if $\delta \geq c_{\ref*{mlb1}}$ then $\tilde{X}(\theta)$ is at least some constant. \end{proof}


\section{Open problems}

There is still a gap in understanding for $t$-sparse vectors. 
\begin{qn} Let $M$ be an $m\times n$ random $t$-sparse matrix. What is the least $N$ such that for all $n\geq N$, the discrepancy of $M$ is at most one with probability at least $1/2$? We know that for $t$ not too large or small, $m \leq N \leq m^3 \log^2 m$. The lower bound is an easy exercise. 
\end{qn}
Next, it would be nice to understand \cref{qn:random_komlos} for more column distributions in other regimes such as $n = O(m)$. In particular, it would be interesting to understand a distribution where combinatorial considerations probably won't work. For example,
\begin{qn} Suppose $M$ is a random $t$-sparse matrix plus some Gaussian noise of of variance $\sqrt{t/m}$ in each entry. Is $\disc M = o(\sqrt{t\log m})$ with high probability? How much Gaussian noise can existing proof techniques handle?
\end{qn}
The quality of the nonasymptotic bounds in this paper depend on the spanningness of the distribution $X$, which depends on how far $X$ is from lying in a proper sublattice of $\cL$. If $X$ actually \emph{does} lie in a proper sublattice of $\cL' \subset \cL$, we may apply our theorems with $\cL'$ instead. This suggests the following:

\begin{qn}
Is there an $N$ depending on only the parameters in \cref{eq:ndisc} \emph{other than spanningness} such that for all $n \geq N$, 
$$\disc M \leq \max_{\cL' \subset \cL} \rho_\infty (\cL')$$
with probability at least $1/2$?
\end{qn}

Next, the techniques in this paper are suited to show that essentially any point in a certain coset of the lattice generated by the columns may be expressed as the signed discrepancies of a coloring. This is why we obtain twice the $\ell_\infty$-covering radius for our bounds. In order to bound the discrepancy, we must know $\rho_\infty(\cL)$. However, the following question (\cref{qn:lattice_komlos} from the introduction) is still open, which prevents us from concluding that discrepancy is $O(1)$ for an arbitrary bounded distribution:
\begin{conj}
There is an absolute constant $C$ such that for any lattice $\cL$ generated by unit vectors, $\rho_\infty(\cL) \leq C$.

\end{conj}
We could also study a random version of the above question:

\begin{qn}\label{qn:random_lattice_komlos}
Let $v_1, \dots, v_m$ be drawn i.i.d from some distribution $X$ on $\R^m$, and let $\cL$ be their integer span. Is $\rho_\infty(\cL) = O(1)$ with high probability in $m$?
\end{qn}
Here we also bring attention to a meta-question asked in \cite{LKP12, HR18}. Interestingly, though we use probabilistic tools to deduce the existence of low-discrepancy assignments, the proof does not yield any obvious efficient randomized algorithm to find them. 
\begin{qn}If an object can be proved to exist by a suitable local central limit theorem, is there an efficient randomized algorithm to find it?
\end{qn}

\section*{Acknowledgements}
The first author would like to thank Aditya Potokuchi for many interesting discussions.

\bibliographystyle{plain}
\bibliography{discrepancy_arxiv}

 \end{document}